\documentclass[11pt]{amsart}
\usepackage{amsfonts, amsmath, amssymb, amscd, amsthm, graphicx, setspace, enumitem, color}
\usepackage{hyperref}
\usepackage{pinlabel}

\newcommand{\ol}[1]{\overline{{#1}}}
\newcommand{\wh}[1]{\widehat{{#1}}}

\newtheorem{thm}{Theorem}[section]
\newtheorem{thmx}{Theorem}

\newtheorem{cor}[thm]{Corollary}
\newtheorem{lem}[thm]{Lemma}
\newtheorem{prop}[thm]{Proposition}

\theoremstyle{definition}
\newtheorem{defn}[thm]{Definition}
\newtheorem{hyp}[thm]{Hypothesis}

\theoremstyle{remark}
\newtheorem{rem}[thm]{Remark}
\newtheorem{ex}[thm]{Example}

\newcommand{\gen}[1]{\left\langle#1\right\rangle}

\newcommand{\prs}[2]{\gen{#1\parallel #2}}

\newcommand{\ball }{\mathbb{B}}
\newcommand{\ga }{\Gamma}

\renewcommand{\d }{{\rm d} }

\renewcommand{\ge }{\geqslant}

\newcommand{\cC }{\mathcal C}

\newcommand{\cH }{\mathcal H}
\newcommand{\cL }{\mathcal L}

\newcommand{\cP }{\mathcal P}

\newcommand{\cS }{\mathcal S}
\newcommand{\cT }{\mathcal T}

\newcommand{\N }{\mathbb N}
\newcommand{\Z }{\mathbb Z}

\newcommand{\fftp}{{\rm fftp}}

\newcommand{\diam}{{\rm diam}}
\newcommand{\dist}{{\rm d}}

\newcommand\Aut{\mathrm{Aut}}

\newcommand\geol{\mathsf{Geo}} 
\newcommand\cone{\mathsf{cone}}

\newcommand{\cte}{M}
\newcommand{\ctee}{M}
\newcommand{\cteee}{R}

\newcommand{\cA}{\mathcal{A}}

\def\coloneq{\mathrel{\mathop\mathchar"303A}\mkern-1.2mu=}

\begin{document}

\title[Counting subgraphs in fftp graphs with symmetry]{Counting subgraphs in fftp graphs with symmetry}

\address{Departamento de Matem\'{a}ticas, Universidad Aut\'{o}noma de Madrid and Instituto de
Ciencias Matem\'{a}ticas, CSIC-UAM-UC3M-UCM
.}
\author{Yago Antol\'{i}n}
\email[Yago Antol\'{i}n]{yago.anpi@gmail.com}

\thanks{The author is supported by the Juan de la Cierva grant IJCI-2014-22425, and acknowledges partial
support from the Spanish Government through grants number MTM2014-54896 and MTM2017-82690-P, and through the “Severo Ochoa Programme for Centres of Excellence in R\&{}D” (SEV-2015‐0554)}
\date{\today}
\begin{abstract}
Following ideas that go back to Cannon, we show the rationality of various generating functions of growth sequences  counting embeddings of convex subgraphs in locally-finite, vertex-transitive graphs with the (relative) falsification by fellow traveler property (fftp). In particular, we recover results of Cannon, of Epstein, Iano-Fletcher and Zwick, and of Calegari and Fujiwara. One of our applications concerns Schreier coset graphs of hyperbolic groups relative to quasi-convex subgroups, we show that these graphs have rational growth, the falsification by fellow traveler property, and the existence of a lower bound for the growth rate  independent of the finite generating set and the infinite index quasi-convex subgroup. 
\end{abstract}



\maketitle



\section{Introduction}
In the celebrated paper \cite{Cannon}, Cannon showed that the growth of  groups acting properly and cocompactly on $\mathbb{H}^n$ is rational, i.e, the generating function of the sequence counting the number of elements in the ball of radius $n$ is a rational function.
His ideas were successively used by Gromov \cite{Gromov} showing that (word) hyperbolic groups have rational growth, by Epstein, Iano-Fletcher and Zwick \cite{EpsteinGrowth} who improved results indicated by Saito \cite{Saito} showing that the generating function  counting the number of embeddings of finite subgraphs in  geodesic automatic Cayley graphs is rational, and recently by Calegari and Fujiwara \cite{CalFuj} obtaining the previous result for vertex-transitive hyperbolic graphs, not necessarily Cayley graphs.

Neumann and Shapiro \cite{NS}  observed that one can use Cannon's arguments under an hypothesis weaker than hyperbolicity, called  the {\it falsification by fellow traveler property} (fftp). 
A graph has this property if there is a constant $\cte$, such that every non-geodesic path $\cte$-fellow travels with a shorter one. 
For Cayley graphs, this property has been widely studied  and there are several examples beyond hyperbolicity. 
The following families of groups have Cayley graphs with fftp for at least one generating set: 
virtually abelian groups and geometrically finite hyperbolic groups \cite{NS}, 
Coxeter groups and groups acting simply transitively on the chambers of locally finite buildings \cite{NoskovCoxeter}, 
groups acting cellularly on locally finite CAT(0) cube complexes where the action is simply transitive on the vertices \cite{NoskovCC}, 
Garside groups \cite{Holt} and Artin groups of large type \cite{HoltRees}. 
The property of having a generating set with fftp is preserved under relative hyperbolicity \cite{AC3} and arguing similarly as in \cite{JJJ} it follows that it is also preserved under graph products.

This paper aims to push the ideas of Cannon to the limit; we will present a common generalization of the previous results for counting ``convex" subgraphs (not necessarily  finite) on locally finite, vertex-transitive graphs\footnote{Let $\ga$ be a graph. We will say that $G\leqslant \Aut(\ga)$ is {\it vertex-transitive} (or that $G$ {\it acts vertex transitively}) if there is a single orbit of vertices of $\ga$ by the action of $G$. If $\Aut(\ga)$ is vertex-transitive, we will say that $\ga$ is vertex-transitive.} (not necessarily Cayley graphs)  with the (relative) falsification by fellow traveler property.  
Graphs in this paper are symmetric directed graph, meaning that all edges are oriented, and if there is an edge from a vertex $v$ to a vertex $u$ there should be an edge from $u$ to $v$. See Section \ref{Sec:not}. In particular, the combinatorial metric on the vertices induced by directed paths coincides with the metric induced by the topological realization where each edge is isometric to the unit interval.

Let $\ga$ be a locally finite, connected, vertex-transitive graph. 
Let $\d_\ga$ denote the combinatorial graph metric on $V\ga$, the vertices of $\ga$. 
Let $Z$ be some graph and let $e_Z(n)$ be the number of different embeddings of $Z$ as a complete subgraph in $\ball_{v_0}(n)$, the  ball of radius $n$ of $\ga$ with center at $v_0$, i.e.
$$e_Z(n)=\dfrac{\sharp \{f\colon Z\to \ga \mid f \text{ injective graph morphism},\, f(Z)\subseteq \ball_{v_0}(n)\}}{|\Aut(Z)|}.$$
For example, if $Z=\bullet$ is a vertex, $e_\bullet(n)$ counts the number of vertices in the ball of radius $n$, and Cannon's  result asserts that  when $\ga$ is the Cayley graph of a group acting properly and cocompactly in $\mathbb{H}^n$,
 $\sum_{n\geq 0} e_\bullet(n) t^n \in \mathbb{Z}[[t]]$
is a rational function. i.e. an element of $\mathbb{Q}(t)$.

In the case $Z$ is infinite, $e_Z(n)$ is equal to zero for all $n$, and to deal with this, we will count embeddings of $Z$ with non-trivial intersection with the ball of radius $n$. However, we need to restrict  to  some family of embeddings. 

\begin{defn}
Let $\ga$ be a graph and  $G\leqslant \Aut(\ga)$ acting vertex transitively. 
A subgraph $Z$ of $\ga$ has {\it $G$-proper embeddings} if  for all $v\in V\ga$, the set  $\{gZ \mid g\in G, v\in gZ\}$ is finite.
\end{defn}

Given $G\leqslant \Aut(\ga)$ and a subgraph $Z$ of $\ga$ with $G$-proper embeddings, we will consider the function $i_{(G,Z)}(n)$ counting the number of $gZ\subseteq \ga$,  $g\in G$, such that $gZ$ {\it intersects} non-trivially the ball of radius $n$ centered at $v_0$. That is
$$i_{(G,Z)}(n)=\sharp \{gZ \mid g\in G, \; gZ\cap \ball_{v_0}(n)\neq \emptyset\}.$$
 In a similar way, we denote by $e_{(G,Z)}(n)$ the number of elements of the $G$-orbit of $Z$ {\it embedded} in the ball of radius $n$. That is $$e_{(G,Z)}(n)=\sharp \{gZ \mid g\in G, \; gZ\subseteq \ball_{v_0}(n)\}.$$
In order to count infinite subgraphs, we will need some mild convexity properties, which we define below. 

\begin{defn}\label{def:proj}
Let $\ga$ be a graph and $Z$ a subgraph. 
For $v\in V\ga$, the closest point projections of $v$ onto $Z$ is denoted by  $$\pi_{Z}(v)\coloneq\{z\in VZ \mid \dist(v,z)=\dist(v, VZ)\}.$$ 

The set $Z$ has {\it $\cte$-fellow projections} if for every $u,v\in V\ga$ with $\d(u,v)=1$ and every $z_v\in \pi_Z(v)$ there exists $z_u\in \pi_Z(u)$ such that $\d(z_v,z_u)\leq \cte$.

The set $Z$ has {\it $\cte$-bounded projections} if for every $u,v\in V\ga$ with $\d(u,v)=1$ the diameter of  $\pi_Z(u)\cup \pi_Z(v)$ is bounded above by $\cte$.
\end{defn}

Note that $\cte$-bounded projections implies $\cte$-fellow projections, and that any finite subgraph $Z$ has both properties.  
We will see other examples with these properties in Section \ref{sec:ex}, including quasi-convex subgraphs of hyperbolic graphs, parabolic subgroups of relatively hyperbolic groups, parabolic subgroups of raags and standard parabolic subgroups of Coxeter groups.

Fix a vertex $v_0$. Let $g\in G$. A {\it $gZ$-geodesic path} is a path from $v_0$ to $gZ$ that realizes the distance $\dist(v_0,gZ)$. We will say that a path $p$ is a {\it $(G,Z)$-geodesic path} if it is a $gZ$-geodesic path for some $g\in G$. Let $g_{(G,Z)}(n)$ be the number of $(G,Z)$-geodesics of length $\leq n$. That is
$$g_{(G,Z)}(n)=\sharp\{p \mid \exists gZ\, \text{ s.t.} p \text{ is geodesic path  from $v_0$ to $gZ$ and }\ell(p)=\dist(v_0,gZ)\}.$$ 

We will introduce the relative falsification by fellow traveler property in Section \ref{sec:paths}. The following, is a particular case of our main result that will be stated in Theorem \ref{thm:mainA+}.
\begin{thmx}\label{thm:mainA}
Let $\ga$ be a locally finite graph with  the falsification by fellow traveler property. 
Let $G\leqslant \Aut(\ga)$ be a vertex-transitive subgroup and  $Z$ a graph with $G$-proper embeddings. If $Z$ has fellow projections, then
\begin{enumerate}
\item[(1)] the $(G,Z)$-geodesic growth function,  $\sum_{n\geq 0} g_{(G,Z)}(n) t^n \in \mathbb{Z}[[t]]$, and
\item[(2)] the $(G,Z)$-embeddings growth function,  $\sum_{n\geq 0} e_{(G,Z)}(n) t^n \in \mathbb{Z}[[t]]$,
\end{enumerate}
are rational functions, and moreover, if $Z$ has  bounded projections, then  
\begin{enumerate}
\item[(3)] the $(G,Z)$-intersections growth function,  $\sum_{n\geq 0} i_{(G,Z)}(n) t^n \in \mathbb{Z}[[t]]$,
\end{enumerate}
is rational.
\end{thmx}
Particular cases previously known of Theorem  \ref{thm:mainA} are: 
(1) for Cayley graphs with the falsification by fellow traveler property and $Z=\bullet$ in \cite{NS}; 
and (2) for $\ga$ hyperbolic and finite graph $Z$ in \cite{CalFuj}; 
using a stronger version of this theorem (Theorem \ref{thm:mainA+}) one recovers (2) for geodesically automatic Cayley graphs and $Z$ finite, which was proved  in \cite{EpsteinGrowth}.

Observe that  Theorem \ref{thm:mainA} can be used to understand the geometry of Schreier coset graphs. 
Indeed, suppose that $\ga$ is a Cayley graph of a group $G$ with respect to some generating set $X$. 
If $Z$ is (the subgraph spanned by) a subgroup, then it has $G$-proper embeddings and $i_{(G,Z)}(n)$  counts how many left cosets of $Z$ intersect non-trivially a ball of radius $n$ in the Cayley graph.
If the generating set $X$ is symmetric, then the previous number is also equal to the number of right cosets meeting the ball of radius $n$, which in turn, is the number of vertices in the ball of radius $n$ of the Schreier coset graph $\ga(G,Z,X)$. 
We will prove

\begin{thmx}\label{thm:mainB}
Let $G$ be a group and $X$ a finite symmetric generating set of $G$.
Suppose that the Cayley graph $\ga(G,X)$ has the falsification by fellow traveler property.
Let $H\leqslant G$ be a subgroup of $G$ such that $H$ (as a subgraph) has fellow projections in $\ga(G,X)$.  
Then:
\begin{enumerate}
\item the Schreier coset graph $\ga(G,H,X)$ has the falsification by fellow traveler property relative to the family of all paths starting at the coset $H$;
\item the set of words $$\geol(H\backslash G,X)=\{w\in X^* \mid \ell(w)\leq \ell(u) \, \forall u\in X^*, u\in_G  Hw\}$$ is a regular language;
\item if moreover, $H$ has  bounded projections, then for  the Schreier coset graph $\ga(G,H,X)$, the function $\sum_{n\geq 0}e_{\bullet}(n)t^n$ is rational. 
\end{enumerate}
\end{thmx}

The Schreier graphs for hyperbolic groups relative to quasi-convex subgroups were studied in \cite{Kapovich} by I. Kapovich. 
There it is proved that if $H$ is an infinite index, quasi-convex subgroup of a non-elementary hyperbolic group $G$, then $\ga(G,H,X)$ is non-amenable. 
A number of consequences are derived from this fact and, in particular, bounds on the co-growth rate are obtained. 

Here, we explore  the growth rate of these Schreier graphs. Using ideas of  \cite{EpsteinGrowth} we obtain
\begin{thmx}\label{thm:mainC}
Let $G$ be hyperbolic, $H$ a quasi-convex subgroup and $X$ a finite symmetric generating set of $G$.
Let $\ga=\ga(G,H, X)$ be the Schreier coset graph  and let $e_{\bullet}(n)=|V\ball_{\ga}(n)|$ be the number of vertices in the ball of radius $n$ centered at $H$ in $\ga$.

There exists a polynomial $Q_X(t)$ depending on $(G,X)$ and a constant $\lambda > 1$ depending only on $G$ such that the following hold:


(1) $Q_X(t)\sum_{n\geq 0} e_{\bullet} (n)t^n$ is a polynomial;

(2) if $G$ is non-elementary and $H$ is of infinite index, then 
$$\limsup \sqrt[n]{e_{\bullet}(n)}\geq \lambda.$$ 
\end{thmx}

\section{Notations, conventions and definitions}\label{Sec:not}
In this paper, a {\it  digraph} $\ga$ is a 4-tuple $(V\ga, E\ga, (\cdot)_-,(\cdot)_+)$, where $V\ga$ is a non-empty set whose elements are called vertices, $E\ga$ is a set, whose elements are called oriented edges, $(\cdot)_-,(\cdot)_+\colon E\ga \to V\ga$ are functions that are called incidence functions. A {\it graph} is a digraph with an involution $(\cdot)^{-1}\colon E\ga\to E\ga$  satisfying that for all $e\in E\ga$, $e_-^{-1}=e_+$ and $e_+^{-1}=e_-$.

A combinatorial path $p$ in a digraph $\ga$  is a sequence $v_0,e_1,v_1,e_2,\dots, e_n,v_n$ where $v_i\in V\ga$, $e_i\in E\ga$ and $(e_i)_-=v_{i-1}$  and $(e_i)_+=v_{i}$. The length of the path $p$ is denoted by $\ell(p)$ and is the number of edges in the sequence.  We extend the adjacent functions to paths, setting $p_-=v_0$ and $p_+=v_n$.  We write $p(i)$ to denote $v_i$ and if $p$ is a path of length $n$, we use the convention that $p(k)=p_+$ for all $k\geq n$. If $\ga$ is a graph, we also extend $(\cdot)^{-1}$ to paths, being $p^{-1}=v_n,e_n^{-1},\dots, e_1^{-1},v_0$.

The combinatorial distance between $u,v$ in $V\ga$, denoted $\d_\ga(u,v)$, is the infimum of the length of the combinatorial paths $p$ with $p_-=u$ and $p_+=v$. A path realizing the combinatorial distance between two vertices is a {\it geodesic.}

In a graph, one should think $e$ and $e^{-1}$ as a single un-oriented edge  and view $\ga$ as metric space that topologically a 1-dimensional CW-complex in which 0-cells are the elements of $V\ga$ and $1$-cells are elements of  $E\ga/(\cdot)^{-1}$, where $[e]=\{e,e^{-1}\}$ is attached to $e_-$ and $e_+$. The metric arises by making each 1-cell isometric to  the interval $[0,1]$ of the real line.  With this setting the combinatorial distance agrees with the induced metric on vertices.

Let  $\lambda \geq 1$ and $c\geq 0$. 
A path $p$ is a $(\lambda, c)$-{\it quasi-geodesic} 
if for any subpath $q$ of $p$ we have 
$$\ell ( q )\leq \lambda \d(q_{-},q_{+})+c.$$
Let $p,q$ be paths in $\ga$  and $\cte\geq 0$. 
We say that $p,q$ {\it asynchronously $\cte$-fellow travel}
if there exist non-decreasing functions $\phi\colon \N\to \N$
and $\psi \colon \N \to \N$ such that
$\d(p(t),q(\phi(t)))\leq \cte$ and $\d(p(\psi(t)),q(t))\leq \cte$ for all $t\in \N$. 
We say that $p,q$ {\it synchronously $\cte$-fellow travel}
if  $\d(p(t),q(t))\leq \cte$ for all $t\in \N$.

Let $G$ be a group, $H$ a subgroup and $X$ a  generating set for $G$. 
The {\it Schreier coset digraph for $G$ relative to  $H$ with respect to $X$} is 
a graph $\ga(G,H,X)$ that has vertex set $H\backslash G$ and edges $(H\backslash G)\times X$ where $(Hg,x)$ is an oriented edge from $Hg$ to $Hgx$ with label $x$. 
If $X$ is symmetric, then $\ga(G,H,X)$ is a graph since there is an edge $(Hgx,x^{-1})$ that we define to be $(Hg,x)^{-1}$.
The  {\it Cayley digraph} of $G$ with respect to $X$ is  $\ga(G,\{1\},X)$  and we just write $\ga(G,X)$. In the case that $X$ is symmetric we say that $\ga(G,X)$ is the Cayley graph.

We use the metric on the Cayley graph to define the length of  $g\in G$ as 
$|g|_X\coloneq \d_X(1,g)$.

\subsection{Regular languages}
\label{sec:regular}
 
A {\it finite state automaton} is a 5-tuple $(\cS,\cA,s_0,X,\tau)$, where $\cS$ is a set whose elements are  called {\it states}, $\cA$ is a subset of $\cS$ of whose states are called {\it accepting states}, a distinguished element  $s_0\in \cS$ called {\it initial state}, a finite set $X$ called the {\it input alphabet} and a function $\tau\colon \cS\times X\to \cS$ called {\it the transition function}.

Let $X$ be a set. We denote by $X^*$ the free monoid generated by $X$. 
Given a non-negative integer $\ctee$, we denote by $X^{\leq \ctee}\subset X^*$ the set of words in $X$ of length at most $\ctee$.

We extend $\tau$ to a function $\tau\colon \cS\times X^*\to \cS$ recursively, by setting $\tau(s,wx)=\tau(\tau(s,w),x)$ where $w\in X^*$, $x\in X$ and $s\in \cS$. 

A {\it language} $\cL$ over $X$ is a subset of $X^*$, and $\cL$ is {\it regular} if there is a finite state automaton $(\cS,\cA,s_0,X,\tau)$
such that $$\cL=\{w\in X^* \mid \tau(s_0,w)\in \cA\}.$$

To a finite state automaton, we can associate a {\it rooted $X$-labeled digraph} $\Delta$, whose vertices are the set of states, the root is the initial state,  and edges are of the form $e=(s,x,\tau(s,x))$ where $e_-=s$ and $e_+=\tau(s,x)$ and the label is $x$. In particular, a word $w\in X^*$ codifies a path in $\Delta$ starting at the root. 

Suppose that $\rho\colon E\Delta \to \mathbb{R}$ is a function. We can extend $\rho$ to a function on paths in $\Delta$ starting at $s_0$ (and hence to words in $X$) by multiplying the values of each edge of the path, i.e if $p=v_0,e_1,v_1,e_2,\dots, v_n$ is path in $\Delta$ starting at the root, we define $\rho(p)=\prod \rho(e_i)$.

\section{Labelling paths in vertex-transitive graphs}\label{sec:paths}
Let $\ga$ be a locally finite, connected, vertex-transitive graph and let $v_0$ be a vertex of $\ga$.

\begin{defn}\label{defn:presentation}
A {\it presentation for the paths starting at $v_0$} is a pair $(X,\theta)$ where $X$ is a  set in bijection  with $(v_0)_-^{-1}=\{e\in E\ga \mid e_- =v_0 \}$,
via $x\mapsto e(x)$ and $e\mapsto x_e$, together with a map  $\theta\colon X\to \Aut(\ga)$ satisfying that for each $x_e\in X$, $\theta(x_e)(v_0)=e_+$.
\end{defn}

Let $X^*$ be the free monoid freely generated by $X$ and $S$ the subsemigroup of $\Aut(\ga)$ generated by $\theta(X)$.
Then $\theta$ extends to a surjective homomorphism $\theta\colon X^*\to S$. 
Let $T$ be the Cayley digraph of $X^*$ with respect to $X$. 
That is $VT=X^*$ and $ET=X^*\times X$ where $(w,x)\in ET$ starts at $w$ and ends at $wx$. 
The maps $w\in X^*=VT\mapsto \theta(w)(v_0)\in \ga$, and $(w,x)\in ET\mapsto \theta(w)(e(x))$ define a digraph homomorphism $\Theta\colon T\to \ga$, and this map is a covering in the following sense: 

{\bf Claim:} for every directed path $p$ in $\ga$ starting at $v_0$ there is a unique path $\widetilde{p}$ in $T$ starting at $1$ such that $\Theta(\widetilde{p})=p$.

Note that it follows from the claim that $\Theta$ is surjective on vertices. 
We will prove the claim by induction on $\ell(p)$, the length of $p$.
If $\ell(p)=1$ then the Claim follows from the definition of presentation of paths.

Now assume that $\ell(p)=n>1$ and the Claim has been established for shorter paths.
Let $f$ be the last edge of $p$ and $p_1$ the initial subpath of $p$ of length $n-1$. 
By induction there is a unique path $\widetilde{p_1}$ in $T$ starting at $1$ with $\Theta(\widetilde{p_1})=p_1$. 
Let $\widetilde{w_1}=(\widetilde{p_1})_+$ and $w_1=\theta(\widetilde{w_1})\in S$.
Then $w_1 v_0=f_-$. Since $w_1\in \Aut(\ga)$, there is some $e\in (v_0)_-^{-1}$ such that $(w_1)^{-1}(f)=e$.
By definition, there is a unique $x\in X$ such that $e(x)=e$ and  $\theta(x)(v_0)=e_+$ and therefore $w_1(e(x))=f$.
Setting $\widetilde{p}$ to be the path $\widetilde{p_1}$ followed by the edge $(\widetilde{w_1},x)$ we obtain the path of the claim. Note that the uniqueness of $\widetilde{p}$ follows from the uniqueness of $\widetilde{p_1}$ and $x$.\hfill\qed 


Since $T$ is a Cayley digraph of a free monoid  freely generated by $X$, directed paths starting at $1$ in $T$ correspond to words over $X$. Conversely, given a path in $\ga$ starting at $v_0$, the lift in $T$ gives a word in $X$ corresponding to the lifted path.
In summary, given $(X,\theta)$,  there is a bijection  between words over $X$ and paths starting at $v_0$ given on words as follows: for a  word  $w\equiv x_1\dots x_n$ in $X$ we assign the combinatorial path $p_w$ consisting on the sequence $v_0,e(x_1),\theta(x_1)(v_0), \theta(x_1)(e(x_2)), \theta(x_1x_2)(v_0),\dots, \theta(x_1\dots x_n)(v_0)$.

\begin{rem}\label{rem: presentation Cayley} 
When $\ga=\ga(G,X)$ is a Cayley graph of a group $G$ with respect to a symmetric generating set there is a canonical presentation of paths starting at $1_G=v_0$, we have the bijection $X\to (v_0)_-^{-1}$, $x\mapsto e(x)=(1,x)$, and $\theta:X\to \Aut(\ga)$, where $\theta(x)$ is the automorphism consisting on acting by $x$ on the left on the elements of $\ga(G,X)$. In this case, the digraph homomorphism $\theta\colon T\to \ga$ preserves the natural labeling on edges of $T$ and $\ga$. 

However, in general, we cannot use the labels of edges of $T$ and the map $\theta$ to define a labeling on the edges of $\ga$ since in general this labeling would not be $\gen{\theta (X)}$-invariant. If it were, then $\ga$ would be the Cayley digraph of $\gen{\theta(X)}$ with respect to $\theta(X)$, however  there are vertex-transitive locally-finite  graphs (for example Diestel Leader graphs \cite{EskinFisherWhyte}) that are not even quasi-isometric to Cayley graphs.
\end{rem}

For sake of notation, we will usually drop the $\theta$ and the function. Thus a word $w$ in $X$ can be seen as an element of $X^*$ or an element of $\Aut(\ga)$ under $\theta$ and we will write $wv_0$ instead of $\theta(w)(v_0)$. The meaning will be clear from the context.

Let us denote the words giving geodesic paths as
$$\geol(\ga,X,\theta)=\{w\in X^* \mid p_w \text{ is a geodesic path}\}.$$
%
%
%
%
%
In general, if $\cP$ if a collection of paths in $\ga$ starting at $v_0$, we denote by $\cL(\cP)=\cL(\cP,X,\theta)$ the language defined by the paths in $\cP$, i.e.
$$\cL(\cP)=\{w\in X^* \mid p_w\in \cP\}.$$

\subsection{Relative falsification by fellow traveler property}

\begin{defn}\label{def:fftp}
Let $\ga$ be a graph, $v_0$ a vertex, $\cP$ a family of paths in $\ga$ and $\ctee\geq 0$.

The family $\cP$ is {\it $v_0$-spanning } if for all $v\in V\ga$ there is a geodesic path $p\in \cP$ from $v_0$ to $v$. 

We denote by $\cP^+$ to the union of $\cP$ and the one-edge continuations of paths in $\cP$, i.e. $\cP^{+}=\cP\cup \{p,e,e_+\mid p\in \cP, e\in E\ga, e_-=p_+\}$.

We say that $\ga$ has the {\it falsification by $\cte$-fellow traveler property relative to $\cP$} ($\ga$ is $\cte$-fftp relative to $\cP$) if every path $p\in \cP^+$  asynchronously $\cte$-fellow travel with a path $q\in \cP$ with the same end points, and moreover if $p$ is not geodesic, then $\ell(q)< \ell(p)$.

%

If $\ga$ is $\cte$-fftp relative to the collection of all paths, we just say that $\ga$ is $\cte$-fftp.
\end{defn}

The original definition of the falsification by fellow traveler property  requires that the paths asynchronously $\cte$-fellow travel, however, it was observed by Elder \cite{Elder02} that the original definition is equivalent to the synchronous definition (up to increasing constants).

The main example of graphs with the falsification by fellow traveler property with respect to a spanning family of paths are Cayley graphs of groups with a generating set admitting a geodesically automatic structure.
 
\begin{ex}\label{ex:relativefftp}
Let $G$ be a group and $X$ a finite generating set for $X$. 
Recall that $\cL\subseteq X^*$ is a {\it geodesic automatic structure} if $\cL$ is a regular language that surjects onto  $G$ via the natural evaluation map, 
each word $w\in \cL$ labels a geodesic path in the Cayley graph $\ga(G,X)$ and there is a constant $\cte$, such that for every $w\in \cL$ and every $x\in X$, there exists $u\in \cL$ such that $wx=_G u$ and the paths labeled by $wx$ and $u$ synchronously $\cte$-fellow travel (see \cite{ECHLPT} for details on geodesic automatic structures).  Clearly $\cP=\{p_w\mid w\in \cL\}$ is a $1_G$-spanning family of paths and $\ga(G,X)$ has the falsification by fellow traveler property with respect to $\cP$.
\end{ex}

We now can state the full version of Theorem \ref{thm:mainA}.

\begin{thm}\label{thm:mainA+}
Let $\ga$ be a locally-finite graph and  $G\leqslant \Aut(\ga)$ be a vertex-transitive subgroup. 
Let $(X,\theta)$ be a presentation for paths in $\ga$ starting at a vertex $v_0$ with $\theta(X)\subseteq G$.
Suppose that $\ga$ has the  falsification by fellow traveler property relative to $\cP$, a $v_0$-spanning collection of paths  with $\cL(\cP)$  regular.
Let  $Z$ be a subgraph of $\ga$ with proper $G$-embeddings and fellow projections.  Then
\begin{enumerate}
\item $\{w\in X^*\mid p_w \text{ is a $(G,Z)$-geodesic}\}\cap \cL(\cP)$ is regular,
\item the $(G,Z)$-embeddings growth function $\sum_{n\geq 0} e_{(G,Z)}(n)t^n$ is a rational,
\item if $Z$ has bounded projections, then the $(G,Z)$-intersections function ${\sum_{n\geq 0} i_{(G,Z)}(n)t^n}$ is rational.
\end{enumerate}
\end{thm}

\section{The \fftp{} Automaton and the n-type directed graph}
 
Throughout this section  $\ga$ is a locally-finite, vertex-transitive graph. 
We start by fixing a vertex $v_0$ of $V\ga$ and $(X,\theta)$ a presentation for paths in $\ga$ starting at $v_0$.  

Let $\ctee\geq 0$. For each $x\in X$, $a\in X^{\leq \ctee}$ with $av_0\in \ball_{xv_0}(\ctee)$, and $b\in X^{\leq \ctee}$ we set $$\d^x(a,b)\coloneq \min\{\ell(p) \mid p_-=a v_0, \; p_+= xb v_0, \; p\subseteq \ball_{xv_0}(\ctee)\} \in [0,2\ctee].$$

\begin{defn}\label{defn:fftp-automaton}
Let $\ctee\geq 0$.  The {\it \fftp{}-automaton} for $(\ga,X,\theta)$ with parameter $\ctee\geq 0$ and accepting states $\cA$, is defined as follows:
\begin{enumerate}
\item[(A1)] the input alphabet is $X$,

\item[(A2)] the set of states $\cS$ consists of  a fail state $\{\varrho\}$ and states of the form  $\phi\colon X^{\leq \ctee}\to [-\ctee,\ctee]$,

\item[(A3)] a distinguished initial state  $\phi_0$ given by $\phi_0(w)=\d(v_0,wv_0)$ for $w\in X^{\leq \ctee},$

\item[(A4)] a transition map $\tau\colon \cS \times X \to \cS$ defined as 
$\tau(\varrho,x)=\varrho$, $\tau(\phi,x)=\varrho$ if $\phi(x)\neq 1$. Otherwise  $\tau(\phi,x)=\psi$ where for 
 $b\in X^{\leq \ctee}$
 $$\psi(b)=\min \{\phi(a)+\dist^x(a,b)-1 \mid a\in X^{\leq \ctee},\, av_0\in \ball_{xv_0}(\ctee)  \}$$
 and the fact that $\psi\in \cS$ will be shown in Proposition \ref{prop:states}, 
\item[(A5)] a subset $\cA$ of $\cS$ of accepting states.
\end{enumerate}

By the {\it fftp-digraph} we will refer to the digraph associated to the fftp automaton.
\end{defn}
In order to simplify the notation, given  $w\in X^*$, we will denote $\tau(\phi_0,w)\in \cS$ by $\phi_w$.

We now clarify the meaning  of the states of the fftp-automaton. 

\begin{prop}\label{prop:states}
Let $(X,\tau, \phi_0, \cS, \cA)$ be the \fftp{} automaton for $(\ga,X,\theta)$ with parameter $\ctee$.
 Then 	 $\phi_w\neq \varrho$ if and only if $p_w$ does not asynchronously $\ctee$-fellow travel with a shorter path with the same endpoints, moreover 
if $\phi_w\neq \varrho$, then for $u\in X^{\leq \ctee}$
\begin{equation}\label{eq:state}
\phi_w(u)=\min\{\ell(q)-\ell(p_w) \mid q_-=v_0, q_+=wuv_0 \text{ and } p_w,q \text{ asyn } \ctee-\text{fellow travel}\}.
\end{equation}
Note that in particular,  if $uv_0=u'v_0$ then $\phi_w(u)=\phi_w(u')$.
\end{prop}

\begin{proof}
We prove the proposition by induction on $\ell(w)$, being the base of induction the case $\ell(w)=0$, where $p_w$ is the  path of length zero starting and ending at $v_0$ and $\phi_w$ is equal to $\phi_0$. 
Observe that the proposition holds in this case.

So assume that the proposition holds for $w$, and we have to show it for $wx$, where $x\in X$. 

Consider first the case $\phi_{wx}=\varrho$. We have two subcases. 
If $\phi_{w}=\varrho$, then by induction hypothesis, $p_w$ asynchronously $\ctee$-fellow travels with a shorter path from $v_0$ to $wv_0$ and so does $p_{wx}$ since  $p_{w}$ is a subpath of $p_{wx}$. 
If $\phi_{w}\neq \varrho$, then $\phi_w(x)\neq 1$ and by \eqref{eq:state} there is a path $q$ from $v_0$ to $wxv_0$ such that $p_w$ and $q$ asynchronously $\ctee$-fellow travel and $\ell(q)-\ell(p_w)<1$. It follows that $\ell(q)<\ell(p_w)+1=\ell(p_{wx})$ and $p_{wx}$ and $q$ asynchronously $\ctee$-fellow travel.

Consider now the case that $\phi_{wx}\neq \varrho$.
Then $\phi_w(x)=1$. 
Suppose that there is a path $q$ from $v_0$ to $wxv_0$ that $\ctee$-fellow travels with $p_{wx}$ and is shorter than $p_{wx}$. 
Then $\ell(q)-\ell(p_w)\leq 0$ and, $q$ also asynchronously $\ctee$-fellow travels with $p_w$. 
Since  $1=\phi_w(x)\leq \ell(q)-\ell(p_w)$ we get a contradiction. 

We now show  that \eqref{eq:state} holds for $\phi_{wx}$. Let $b\in X^{\leq \ctee}$.
By definition (A4), $\phi_{wx}(b)=\min \{\phi_w(a)+\dist^x(a,b)-1 \mid a\in X^{\leq \ctee},\, av_0\in \ball_{xv_0}(\ctee)\}$. 

Let $a\in X^{\leq \ctee}$ realizing the minimum in (A4). Then, by induction hypothesis, there exists a path $q_1$ from $v_0$ to $wav_0$ such that $\ell(q_1)-\ell(p_w)=\phi_{w}(a)$, $(q_1)_+\in \ball_{wxv_0}(\ctee)$ and $q_1$ and $p_{wx}$ asynchronously $\ctee$-fellow travel. 
By definition of $\dist^x$, there is a path $q_2$ from $wav_0$ to $wxbv_0$ of length $\dist^x(a,b)$ that asynchronously $\ctee$-fellow travel with $wxv_0$. 
Thus $q=q_1q_2$ asynchronously $\ctee$-fellow travels with $p_{wx}$ and $\ell(q)-\ell(p_{wx})=\ell(q_1)-\ell(p_{w})+\ell(q_2)-1$. 
We have shown that $\phi_{wx}(b)$ is greater or equal than the right-hand side of \eqref{eq:state}.

Let $q$ be a  path realizing the minium in the right-hand side of  \eqref{eq:state}. That is $q$ is a path of minimal length from $v_0$ to $wxbv_0$ that asynchronously $\ctee$-fellow travels with $p_{wx}$. 
Write $q$ as $q_1q_2$ where  $q_1$ be the longest initial subpath of $q$ that asynchronously $\ctee$-fellow travel with $p_w$, and $q_2$ might be an empty subpath. 
Note that $q_2\subseteq \ball_{wxv_0}(\ctee)$ and that $(q_1)_+\in \ball_{wxv_0}$. By the minimality of $q$ one has that $\ell(q_1)-\ell(p_w)=\phi_w(a)$ where $a\in X^{\leq \ctee}$ is such that $wav_0=(q_1)_+$ and $\ell(q_2)=\d^x(a,b)$. 
Thus we get that $\ell(q)-\ell(p_{wx})=\ell(q_1)+\ell(q_2)-\ell(p_{w})-1=\phi_w(a)+\dist^x(a,b)-1\geq \phi_{wx}(b)$. 
Thus $\phi_{wx}(b)$ is less or equal than the right-hand side of \eqref{eq:state}. 
This completes the proof.
\end{proof}

\begin{rem}
The transition function $\tau$ in Definition \ref{def:fftp} does not completely agree with the one of Neumann and Shapiro of \cite[Proposition 4.1]{NS}. 
Their transition function takes the minimum over the $a\in X^{\leq M}$ such that $\dist(av_0,xbv_0)\leq 1$ (which is equivalent to $\dist^x(a, b)\leq 1)$, and later it is claimed without proof that  Proposition \ref{prop:states} holds for that automaton. 
Our definition of the transition function for the fftp-automaton resembles more to the one of \cite[Lemma 3.7]{CalFuj} where essentially it describes the transition between different tournaments. 
We remark that it is not exactly equal because without some extra structure (such as hyperbolicity) we do not know that $\dist^x(a,b)$ agrees with $\dist(av_0,xbv_0)$.
\end{rem}

\subsection{Fftp graphs and the \fftp{}-automaton}
We will now show the power of this automaton when $\ga$ is a \fftp{} graph. 

\begin{hyp}\label{hyp:fftp}
Let $\ga$ be a vertex transitive graph, $v_0\in V\ga$, and $(X,\theta)$ a presentation of paths starting at $v_0$. 
Let $\cP$ be a $v_0$-spanning family of paths and assume that $\ga$ is $\ctee$-\fftp{} relative to $\cP$. 
Let $(X, \tau, \phi_0, \cS, \cA)$ be the \fftp{} automaton for $(\ga,X,\theta)$ with parameter $\ctee^2$.
\end{hyp}

The importance of $\ctee^2$ will be evident soon.
First, we need the following fact, whose proof consist of repeatedly using the definition.
\begin{lem}\label{lem:ftqg}
With Hypothesis \ref{hyp:fftp}.
Let $B\geq 1$.
Let $p$ be a concatenation of a geodesic path $p'$ in $\cP$ starting at $v_0$ and a path of length at most $B$ starting at $p'_+$. Then there exists a geodesic path $q$ in $\cP$ with the same endpoints as $p$ and such that $p$ and $q$ asynchronously  $B\cdot \ctee$-fellow travel.
\end{lem}

\begin{prop}\label{prop:accept_geo}
With Hypothesis \ref{hyp:fftp}. Let $w\in \cL(\cP)$. Then $\phi_w\neq \varrho$ if and only if $p_w$ is geodesic. Moreover, if $\phi_w$ is not the fail state, then for all $u\in X^{\leq \ctee}$, 
\begin{equation}\label{eq:timedif}
\phi_w(u)=\d_\ga(v_0,wuv_0)-\d_\ga(v_0,wv_0).
\end{equation}
\end{prop}
\begin{proof}
Let $w\in \cL(\cP)$. If $p_w$ is geodesic, then it cannot asynchronously fellow travel with a shorter path, and by Proposition \ref{prop:states}, $\phi_w\neq  \varrho$. If $p_w$ is not geodesic, by the definition of \fftp{} relative to $\cP$, there exists a shorter path that asynchronously $\ctee$-fellow travels with $p_w$ and thus $\phi_w=\varrho$.

Now assume that $\phi_w\neq \varrho$ and let $u\in X^{\leq \ctee}$. The path $p_{wu}$  is a concatenation of a geodesic $p_w$ and a path of length at most $\ctee$, and hence by Lemma \ref{lem:ftqg} it asynchronously $\ctee^2$-fellow travels with a geodesic $q$, and \eqref{eq:timedif} follows from \eqref{eq:state}.
\end{proof}

It follows from Proposition \ref{prop:accept_geo} that the fftp-automaton accepts those $w\in \cL(\cP)$ that are geodesic. 
Hence the language $\geol(\ga,X,\theta)\cap \cL(\cP)$ is regular (this will be a particular case of Corollary \ref{cor:regularZgeo}). 
Therefore, the generating function of the number of geodesic paths in $\cP$ is a rational function. 

We aim to use the fftp-automaton  not just to count geodesic paths of length $\leq n$ but also to count  the number of vertices in $\ball_{v_0}(n)$. 
It will be convenient to understand which states accept geodesics ending in the same vertex.

\begin{defn}[Cannon's $N$-type]
Let $G\leqslant \Aut(\ga)$. 

The {\it $N$-type} of a vertex $v$ of $\ga$ is the function $\kappa_v\colon \ball_v(N)\to [-N,N]$ given by  $$\kappa_v(u)=\dist(v_0,u)-\dist(v_0,v).$$

Two vertices $v$ and $v'$  of $\ga$ have the {\it same $N$-type  mod $G$}, if there is an automorphism $\alpha\in G$ of $\ga$, such that $\alpha(v)=v'$,  and
$$\d(v_0, u)-\d(v_0,v)=\d(v_0, \alpha(u))-\d(v_0,u')$$
for every $u\in V\ga$ with  $\d(v,u)\leq N$. 
\end{defn}
To have the same $N$-type mod $G$ defines an equivalence relation on the vertices of $\ga$, that we denote by $\sim_N$, and by {\it $N$-type mod $G$} of a vertex $v$ we refer to its equivalence class.

\begin{defn}
Let $G\leqslant \Aut(\ga)$ and $m\geq M$. 
Two functions $\phi,\phi'\colon  \ball_{v_0}(m)\to [-m,m]$ are {\it $\ctee$-equivalent mod $G$} if  there is  $\alpha\in G$  such that $\phi'\vert_{\ball_{v_0}(\ctee)}=(\phi \circ \alpha)\vert_{\ball_{v_0}(\ctee)}$, i.e. they agree in the ball of radius $\ctee$ up to the automorphism $\alpha$. 
\end{defn}
If $\phi$ and $\phi'$ are $\cte$-equivalent mod $G$, we write $\phi\sim_\cte\phi'$.

It follows from Proposition \ref{prop:states} that the states of the \fftp{}-automaton  $\phi_w\colon X^{\leq \ctee^2}\to [-M^2,M^2]$ induce a function $\ball_{v_0}(\ctee^2)\to [-M^2,M^2]$, $uv_0\mapsto \phi_w(u)$. 
For the sake of simplifying the notation, we will also denote this function by $\phi_w$.

%
%
%
%
%

The following will be key in the rest of the paper, since it allows us to read the $\cte$-type of a vertex $v$ at the state of fftp-automaton accepting a geodeic $p_w$ from $v_0$ to $v$.

\begin{lem}\label{lem:k-equivalent}
With Hypothesis \ref{hyp:fftp} and assuming that $\theta(X)\subseteq G$.
Let $w,w'\in X^*$ and suppose that $p_w$ (respectively $p_{w'}$) is a geodesic path from $v_0$ to $v$ (respectively $v$'). Then $\phi_w$ and  $\phi_{w'}$ are $\ctee$-equivalent mod $G$ if and only if $v$ and $v'$ have the same $\cte$-type mod $G$.

In particular, if $v=v'$ then $\phi_w$ and  $\phi_{w'}$ are $\ctee$-equivalent mod $G$.
\end{lem}

\begin{proof}
Suppose fist now that $v\sim_\cte v'$ mod $G$. Then, there is $\alpha\in G$ such that $\alpha(v)=v'$ and for all $u\in \ball_v(\ctee)$
$$\dist(v_0,u)-\dist(v_0,v)=\dist(v_0,\alpha(u))-\dist(v_0,v').$$
For convenience, we will rewrite the previous facts as $\alpha(wv_0)=w'v_0$ and for all $u\in X^{\leq \cte}$
\begin{equation} \label{eq:equiv} \dist(v_0,wuv_0)-\dist(v_0,wv_0)=\dist(v_0,\alpha(wuv_0))-\dist(v_0,w'v_0).\end{equation}

Seeing $w$ and $w'$ as automorphism of $\ga$, we let $\beta= (w')^{-1}\circ \alpha \circ w$. Since $\theta(X)\subseteq G$, $\beta\in G$. We have that $\beta(v_0)=(w')^{-1}\alpha(wv_0)=(w')^{-1}v'=v_0$ and for every  $u\in X^{\leq \ctee}$
\begin{align*}
\phi_w(uv_0)&=\d(v_0,w u v_0)-\d(v_0,wv_0)& \text{by Proposition \ref{prop:accept_geo}}\\
&=\d(v_0,\alpha(wuv_0))-\d(v_0,w'v_0)&\text{by equation \eqref{eq:equiv}}\\
&= \dist(v_0, w' (w')^{-1}\alpha(wuv_0))-\dist(v_0,w'v_0)&\\
&= \dist(v_0,w'\beta(uv_0))-\dist(v_0,w'v_0) & \\
&= \phi_{w'}(\beta(uv_0))& \text{by Proposition \ref{prop:accept_geo}}
\end{align*}

Similarly, suppose now that $\phi_w\sim_{\cte}\phi_{w'}$. Then  there $\alpha \in G$ such that for all $u\in X^{\leq \cte}$, $\phi_w(uv_0)=\phi_{w'}(\alpha(uv_0))$.
Take $\beta=w'\circ \alpha \circ w^{-1}$ and note that $\beta(wv_0)=w'(\alpha v_0)= w' v_0$. 

Now, for $wuv_0\in \ball_{wv_0}(\ctee)$ we have that 
\begin{align*}
\dist(v_0,wuv_0)-\dist(v_0,wv_0)=\phi_{w}(uv_0)&=\phi_{w'}(\alpha(uv_0))\\
&=\dist(v_0,w'\alpha(u v_0))-\dist(v_0,w'v_0)\\
&=\dist(v_0,\beta(wu v_0))-\dist(v_0,w'v_0).
\end{align*}
Thus $wv_0$ and $w'v_0$ have the same $\ctee$-type mod $G$.
\end{proof}

\subsection{Random geodesic combings}
\begin{defn}
A {\it random geodesic combing} of a graph $\ga$ is a set of probability measures $\{\mu_v\mid v\in V\ga\}$ where each $\mu_v$ has support on the set of geodesic paths starting at $v_0$ and ending at $v$. A {\it geodesic combing} is a random geodesic combing where each probability measure has support on a single geodesic path. 

Since paths starting at $v_0$ are codified by words in $X$,  we can think that $\mu_v$ is a probability measure defined on $X^*$, whose support is contained in the set 
$\{w\in \mathsf{Geo}(\ga,X,\theta) \mid wv_0=v\}$.  
We will say that a  random geodesic combing $\{\mu_v\mid v\in V\ga\}$ is  {\it Markov} if there is a finite rooted $X$-labeled digraph $\Delta$ and a function $\rho\colon E\Delta \to [0,1]$, such that for every $w\in X^*$, $\mu_{wv_0}(p_w)=\rho(w)$, where $\rho$ is extended to (labels of) paths in $\Delta$ as in subsection \ref{sec:regular}.
\end{defn}

For future use, we will need count vertices of a certain $M$-type in a ball of radious $n$. 
\begin{defn}\label{defn:auxilariy e}
Let $\sim_{\ctee}$ denote the equivalence relation on the vertices of $\ga$ consisting on having the same $\ctee$-type mod $G=\gen{\theta(X)}$. Let $$e_{(G,\bullet)}^{[v]}(n)\coloneq \sharp \{u\sim_\ctee v \mid \dist(v_0,u)\leq n \}$$ and $$\ol{e}_{(G,\bullet)}^{[v]}(n)\coloneq \sharp \{u\sim_\ctee v \mid \dist(v_0,u)= n \}.$$
\end{defn}

\begin{thm}\label{thm:Markov}
With Hypothesis \ref{hyp:fftp} and assuming that $G=\gen{\theta(X)}$. The following hold:
\begin{enumerate}
\item[(i)] Let $\Delta$ be the fftp-digraph.
There is a function $\rho\colon E\Delta \to [0,1]$ defining a Markov random geodesic combing for the graph $\ga$.

\item[(ii)]  There is a square non-negative matrix $\mathbf{A}$, and for every $v\in V\ga$ there are non-negative vectors $\mathbf{u}$ and $\mathbf{v}$ such that $\overline{e}^{[v]}_{(G,\bullet)}= \mathbf{v}^T\mathbf{A}^n\mathbf{u}$. 

\item[(iii)] For every $v\in V\ga$, the series 
$$\sum_{n\geq 0} e_{(G,\bullet)}^{[v]}(n)t^n \text{ and }\sum_{n\geq 0} \ol{e}_{(G,\bullet)}^{[v]}(n)t^n  $$
are rational functions.
\end{enumerate}

\end{thm}

\begin{proof}

(i). 
Let $v_0,v\in V\ga$. A {\it $v_0$-parent of $v$} is an edge $e$ adjacent to $v$ in a geodesic from $v_0$ to $v$.
If $w$ and $w'$ are words accepted in equivalent states of the \fftp{}-automaton, then $wv_0$ and $w'v_0$ have the same number of parents, which is the number of $x\in X$ for which $\phi_w(x)=-1$.


We assign weights on the edges of the fftp-digraph $\Delta$ as follows: for an edge going from a state  $\phi_{w}$ to a state $\phi_{wx}\neq \varrho$ we assign the weight $\frac{1}{\sharp \text{parents } \phi_{wx}}$; for an edge going to the state $\varrho$ we  assign weight $0$. 
We now extend this weights to paths in $\Delta$ by multiplying the weights along a path. 
Denote by $\rho$ this weight function on paths. 
For convenience,  the path of length zero has weight $1$. 

We get that each path in $\Delta$ starting at $\phi_0$ and not ending in $\varrho$ codifies a word labeling a geodesic path in $\ga$  starting at $v_0$, and it has an associated weight. So we can think that $\rho$ is a function on paths on $\ga$ starting at $v_0$ that assigns a weight with values in $[0,1]$ and that a necessary and sufficient condition for $\rho(p)=0$ is that $p$ is a non-geodesic path in $\ga$.

We need to show that  for each  $v\in V\ga$, $\sum_{p \in \mathsf{Geo}(v_0,v)}\rho(p)=1$. The proof is by induction on $\d(v_0,v)$. 
The base of induction is $\dist(v_0,v)=1$.
Consider the elements $x_1,\dots, x_n$ in $X$ satisfying that  $x_iv_0=v$. 
Then, by Lemma \ref{lem:k-equivalent}, the functions have the same type, i.e $[\phi_{x_i}]=[\phi_{x_j}]$ for $1\leq i,j\leq n$ and $\phi_{x_i}(y)=-1$ if and only if $y\in \{x_1^{-1},\dots ,x_n^{-1}\}$. Thus the number of parents of $v$ is $n$ and that is the number of paths in $\Delta$ that codify a geodesic path from $v_0$ to $v$ and hence, each of this paths have weight $\frac{1}{n}$ and the claim follow in this case.

Now assume that the result has been proved for  $\d(v_0,v')\leq m$ with $m\geq 1$ and assume that $\d(v_0,v)=m+1$. Let $v_1,\dots, v_k$  be the vertices in $\ga$ with $\d(v_i,v)=1$ and $v_i$ in some geodesic from $v_0$ to $v$. Let $e_{i,j}$, $i=1,\dots,k$ and $j=1,\dots,n(j)$ be edges from $v_i$ to $v$.
Notice that these are all the parents of $v$, and therefore they correspond to an edge $\hat{e}_{i,j}$ in the fftp-automaton of weight $\frac{1}{\sharp \text{ parents of v}}$. 

Then \begin{align*}
\sum_{p\in \mathsf{Geo}(v_0,v)}\rho(p) & = \sum_{i=1}^k  \sum_{j=1}^{n(j)} \rho(\hat{e}_{i,j}
) \sum_{p\in \mathsf{Geo}(v_0,v_i)} \rho(p)\\
& =\sum_{i=1}^k  \sum_{j=1}^{n(j)} \rho(\hat{e}_{i,j})=\left(\sharp \text{ parents of v}\right) \frac{1}{\sharp \text{ parents of v}}=1.
\end{align*}

(ii). We continue with the notation of (i).

Let $\mathbf{M}=(m_{i,j})$ be the transition of matrix the fftp-digraph i.e. $m_{ij}$ is the number of edges starting in the state $i$ and ending in the state $j$. 
We let the index $0$ correspond to the state $\phi_0$.  
Form a new matrix $\mathbf{A}$ where for $j\neq \varrho$,  ${a}_{ij}$ is equal to $m_{i,j}$ divided by the number of parents of the state corresponding to the index $j$. 
Using the theory of Markov Chains, it follows that if $a_{0j}(n)$ denotes the $(0,j)$-term of $\mathbf{A}^n$, we have that $a_{0j}(n)$ is the sum of the weights of the paths in $\Delta$ of length $n$ starting at $\phi_0$ and ending in the state $j$, which correspond to the number of geodesic paths in $\ga$ starting a $v_0$ and accepted at the state $j$.

Let $v\in V\ga$ and let $\cT$ a set of those $\phi_w\neq \rho$ with $wv_0\sim_\ctee v$ i.e. the states corresponding to paths ending in vertices with the same $M$-type as $v$. 
Let $\mathbf{u}$ be the column vector with all entries equal to zero except from the entry corresponding to $\phi_0$ which is equal to 1. Similarly, let  $\mathbf{v}$ be the column vector with $v_i=1$ if $i\in \mathcal{T}$ and $v_i=0$ otherwise. 
Then for $n\geq 1$,
$\mathbf{v}^T \mathbf{A}^n \mathbf{u}=\sum_{j\in \cT}a_{0j}(n)$  represents the sum of the weights of the paths in $\ga$ starting at $v_0$ and finishing at a vertex at distance $n$ with the same $\ctee$-type as $v$. 

(iii). We continue with the notation of (ii). First notice that 
\begin{equation}
\label{eq:eole}(1-t)\sum_{n\geq 0} e_{(G,\bullet)}^{[v]}(n)t^n=\sum_{n\geq 0}  \ol{e}_{(G,\bullet)}^{[v]}(n)t^n.
\end{equation}

So to prove the rationality of the series, it will be enough to show the rationality of the right-hand side one

Thus $$\sum_{n\geq 0}  \ol{e}_{(G,\bullet)}^{[v]}(n)t^n = \sum_{n\geq 0 }\mathbf{v}^T \mathbf{A}^n \mathbf{u} t^n= \mathbf{v}^T (Id - \mathbf{A}t)^{-1} \mathbf{u}= \dfrac{\mathbf{v}^T\mathrm{Adj}(Id-\mathbf{A}t)\mathbf{u}}{\det(Id-\mathbf{A}t)} $$
which is a quotient of two polynomials in $\mathbb{Q}[t]$.
\end{proof}

\subsection{Cannon's proof: Cone-types and \texorpdfstring{$N$}{N}-types}
The fftp-digraph used in Theorem \ref{thm:Markov} might be unnecessarily big.
In fact, Cannon \cite{Cannon} shows that under hyperbolicity, for $N$ sufficiently big, the $N$-types  determine the cone types and hence there are finitely many cone types. 
The rationality of the growth series follows from a recursion involving the cone types, which is similar to the argument in the proof of Theorem \ref{thm:Markov}. Neumann and Shapiro \cite{NS} realized that one just needs fftp to show that for $N$ sufficiently big, the $N$-type determine the cone type.
In this subsection,  for the sake of completeness, we will see that there are covering digraph maps from the fftp-digraph to the $N$-type digraph, and from the $N$-type digraph to the cone-type digraph.

\begin{lem}\label{lem:fftp-digraph}
With Hypothesis \ref{hyp:fftp}.
Let $(X,\tau, \phi_0, \cS,\cA)$ be the \fftp{}-automaton for $(\ga,X,\theta)$.
Let $\Delta$ be the fftp-digraph.
The projection map $\mathcal{S}\to \mathcal{S}/\sim_\ctee$ extends to a digraph map  $\pi\colon\Delta\to\ol{\Delta}$ such that for any two paths in the fftp-digraph that represent geodesic paths to the same vertex of $\ga$, their projection to $\ol{\Delta}$ end in the same state.
\end{lem}
We will say that  $\ol{\Delta}$ is a  {\it $\ctee$-type digraph.}
\begin{proof}
We note that the canonical projection map on vertices given by $\phi \mapsto [\phi]$, does not  extend to a canonical map on edges in general, and the construction will depend on the choices of automorphism realizing the $\ctee$-equivalence between vertices of $\Delta.$

We fix a representative $\psi_0$ for each $\ctee$-equivalent class of vertices $[\psi]$, and for each $\psi\in [\psi_0]$ we  fix an automorphism $\alpha_{\psi}\in G$ such that $\psi\mid_{\ball_{v_0}(\ctee)}=(\psi_0\circ\alpha_{\psi})|_{\ball_{v_0}(\ctee)}$. 
Recall the notation of Definition \ref{defn:presentation}. We have a bijection from $X$ to $(v_0)_{-}^{-1}$ given by $x\mapsto e(x)$ and with inverse map $e\mapsto x_e$.
Let $x,x'\in X$ such that $e_x=\alpha(e_{x'})$. 
We map the edges $f$ of $\Delta$, with $f_-= \psi\in [\psi_0]$ and $f_+=\tau(\psi, x')$ to the same edge $\ol{e}\in \Delta$ from $[\psi_0]$ to $[\tau(\psi_0,x)]$. By construction, this map is surjective and a local isomorphism, and thus a covering map.

The fftp-digraph is an $X$-labeled graph. If two words $w$, $w'\in X^*$ represent geodesics paths starting at $v_0$ and ending at $v$, then by Lemma \ref{lem:k-equivalent} 
the vertices $\phi_w$ and $\phi_{w'}$ of $\Delta$ are $\ctee$-equivalent mod $G$.
\end{proof}

\begin{rem}
In the case of Cayley graphs, the group of automorphism considered preserves the labeling of the edges and hence two functions are in the same $\ctee$-type if and only if they are equal. In this case the covering map $\Delta\to \ol{\Delta}$ is canonical.
\end{rem}

\begin{defn}
Let $\ga$ be a vertex-transitive graph and $v_0$ a vertex.

The {\it cone} of $v\in V\ga$, denoted $\cone(v)$, is the set paths $p$ starting at $v$ that continue a geodesic from $v_0$ to $v$.
Two vertices $v$ and $v'$ have the same {\it cone type mod $G$} if there is an automorphism of $\alpha\in G$ such that $\alpha(v)=v'$ and $\alpha(\cone(v))=\alpha(\cone(v'))$.
\end{defn}

The following is an standard argument and goes back to Cannon \cite{Cannon}. 
\begin{lem}\label{lem:cone_types}
With the hypothesis \ref{hyp:fftp} and assume further that $\cP$ is the collection of all paths. 
If $\phi_w$ and $\phi_u$ are $\ctee$-equivalent mod $G$ then $(p_w)_+,$ and $(p_{u})_+$ have the same cone type mod $G$.
\end{lem}

\begin{proof} 
Let $v=(p_u)_+$ and $v'=(p_w)_+$.
Suppose that there is an automorphism $\beta\in G$  satisfying that $\beta(v)=v'$ and $\phi_w(\beta(a))= \phi_u(a)$ for all $a\in X^{\leq \ctee}$.
We will prove $\beta(\cone(v))=\cone(v')$.
Note that by symmetry, it is enough to show that $\beta(\cone(v))\subseteq \cone(v')$. 

Suppose that there is a path $q$ in $\cone(v)$ such that $\beta(q)$ does not belong to $\cone(v')$. 
Without loss of generality, we can assume that $\ell(q)$ is minimal with this property.
Let $t$ be path consisting on $p_w$ followed by $\beta(q)$. 
Then $t$ is a minimal non-geodesic  i.e. a path for which all proper subpath is geodesic, and hence it $\ctee$-fellow travels with a geodesic path $r$. 
Let $r$ be broken into two subpaths $r_1$ from $v_0$ to some point $a$ at distance at most $\ctee$ from $v'$ and $r_2$ from $v$ to $\beta(q)_+$. 
Let $s_1$ be a geodesic path in from $v_0$ to $\beta^{-1}(a)$ and $s_2$ be $\beta^{-1}(r_2)$. 
Let $s$ be the concatenation of $s_1$ and $s_2$. 
See Figure \ref{fig:cone_type}.

\begin{figure}[ht]
\labellist
\pinlabel $v_0$ at  140 -12
\pinlabel $v$ at -8 195
\pinlabel $\beta^{-1}(a)$ at 82 206
\pinlabel $q\in \cone(v)$ at -22 235
\pinlabel $\beta(q)$ at 191 190
\pinlabel $v'=\beta(v)$ at 163 148
\pinlabel $a$ at 235 130 
\pinlabel $r_2$ at 250 175
\pinlabel $s_2=\beta^{-1}(r_2)$ at 90 240 
\pinlabel $p_u$ at 65 80
\pinlabel $p_w$ at 155 80
\pinlabel $s_1$ at 103 100
\pinlabel $r_1$ at 180 29
\endlabellist
\centering
\includegraphics[scale=0.8]{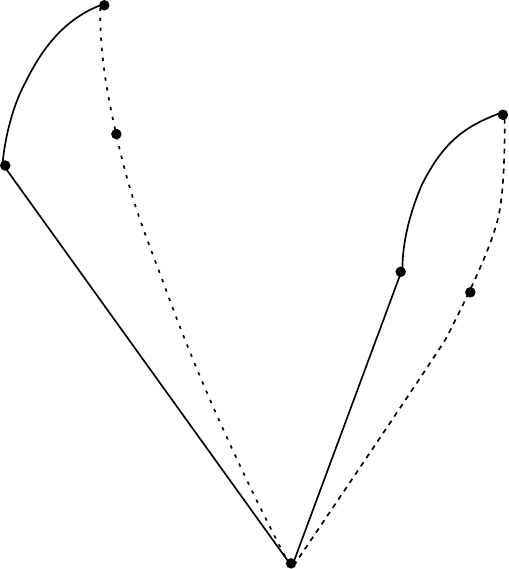}
\caption{Paths in the proof of  Lemma \ref{lem:cone_types}.}
\label{fig:cone_type}
\end{figure}
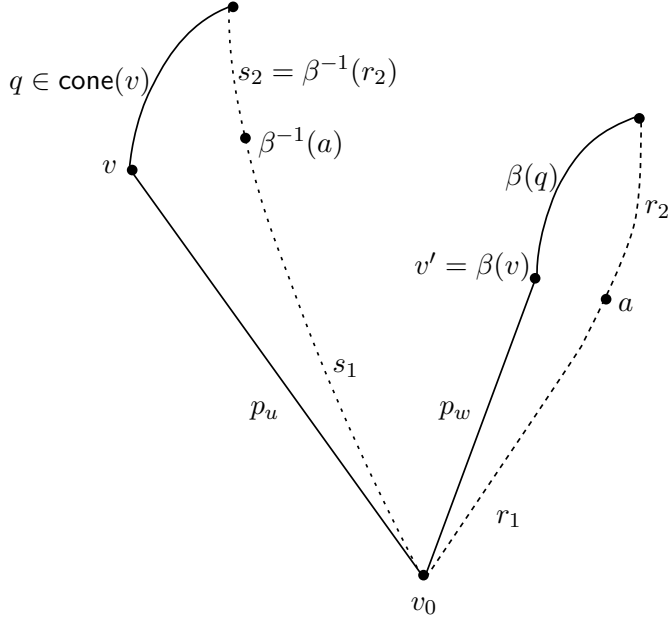
By \eqref{eq:timedif}, $$\ell(s_1)-\ell(p_u)=\phi_{u}(\beta^{-1}(a))=\phi_{w}(a)=\ell(r_1)-\ell(p_w).$$
Then $\ell(s)-\ell(p_u)-\ell(q)=\ell(s_1)-\ell(p_u)+\ell(s_2)-\ell(q)=\ell(r_1)-\ell(p_w)+\ell(r_2)-\ell(\beta(q))=\ell(r)-\ell(t)=-1$. 
Thus $p_u$ followed by $q$ is not geodesic, contradicting that $q\in \cone(v)$.
\end{proof}

Assume the hypothesis of the previous lemma.
Then $\ga$ has finitely many cone types, and can we construct the {\it cone-type digraph} $\overline{\overline\Delta}$ as follows. The vertices are  the different cone types (i.e. $V\overline{\overline\Delta}=\{ \cone(v) \mid v\in V\ga\}$) and and there is an edge from $\cone(v)$ to $\cone(u)$ if there is an edge in $\cone(v)$ starting at $v$ and ending at $u'$ with $\cone(u')$ equivalent to  $\cone(u) \mod G$. 
Thus, given a covering $\pi\colon \Delta\to \ol{\Delta}$ from the fftp-digraph to an $\ctee$-type digraph, we have a natural map from $V\ol{\Delta}$ to $V\overline{\overline\Delta}$ and it is easy to extend it to a digraph covering map $\ol{\Delta}$ to $\overline{\overline\Delta}$ by using $\pi$ to determine to which type of edge in $\ga$ correspond traversing a given edge in $\ol{\Delta}$.

\begin{rem}
It is worth recalling the example of Elder \cite{ElderNofftp} of a virtually abelian group with a generating set that does not have the falsification by fellow traveler property, but it has finitely many cone-types. 
Thus fftp is strictly stronger than having finitely many cone types. 
\end{rem}

We finish this section by recalling that there is also the concept of the {\it $k$-tail} of a vertex, which is similar to the $k$-type but records the vertices in a ball of radios $k$ at $v$ that cannot be reached by a geodesic path from $v_0$ passing through $v$. See \cite[III.$\ga$, Theorem 2.18]{BH}.



\section{Choosing the accepting states and proof of Theorem \ref{thm:mainA+}}\label{sec:proofA}

In this section we prove Theorem \ref{thm:mainA+}. Throughout this section we assume Hypothesis \ref{hyp:fftp} and the hypothesis of Theorem \ref{thm:mainA+}. That is, $\ga$ is a locally finite, $G\leqslant \Aut(\ga)$ is vertex transitive, $v_0\in V\ga$, $(X,\theta)$ is a presentation for paths starting at $v_0$ with $\theta(X)\subseteq G$, $\cP$ is $v_0$-spanning family of paths, $\cL(\cP)$ is regular, $\ga$ has $\cte$-fftp relative to $\cP$, and $Z$ is a subgaph of $\ga$ with $\cte$-fellow projections $G$-proper embeddings.

\begin{lem}\label{lem:fftpcoset}
%
For every path $p$ in  $\cP$ from $v_0$ to $Z$ such that  $\ell(p)>\d(v_0,Z)$, there is a path $q\in \cP$ from $v_0$ to $Z$  such that $\ell(q)< \ell(p)$,  $\d(p_+,q_+)\leq \cte$, and $p$ and $q$ asynchronously $\cte^2$-fellow travel.
\end{lem}

\begin{proof} Let $p\in \cP$ from $v_0$ to $Z$ such that $\ell(p)>\d(v_0,Z)$.

If $p$ is not geodesic, then by \fftp{} there is a shorter path  $q\in \cP$ with the same endpoints that asynchronously $\cte^2$-fellow travel with it and the claims of the lemma are satisfied.

So assume that $p$ is geodesic.
Recall that  $p(i)$, $i=0,1,\dots, \ell(p)$ denotes the $i$th vertex in the combinatorial path defined by $p$.
Since  $\ell(p)> \d(v_0,Z)$, $p_+\notin \pi_{Z}(v_0)=\pi_{Z}(p_-)$. 
Let $j=\max\{i \mid p_+\notin \pi_{Z}(p(i))\}$, that is $p(j)$ is the closest vertex of $p$  whose projection to $Z$ does not contain the vertex $p_+$. See Figure \ref{fig:coset}.
Note that $j\neq \ell(p)$ and by definition of $j$ we have that $p_+\in \pi_Z(p(j+1))$. 
Since $\d(p(j),p(j+1))=1$ by the fellow projections, there is $z_{j}\in \pi_Z(p(j))$ such that $\d(z_{j},p_+)\leq \cte$. 
Let $r$ be a path from $p_+$ to $z_j$. 
Since $p\in \cP$ is geodesic and $\ell(r)\leq \cte$, by Lemma \ref{lem:ftqg}, the concatenation of $p$ and $t$ asynchronously $\cte^2$-fellow travel with a geodesic path $q\in \cP$ from $v_0$ to $z_j$.
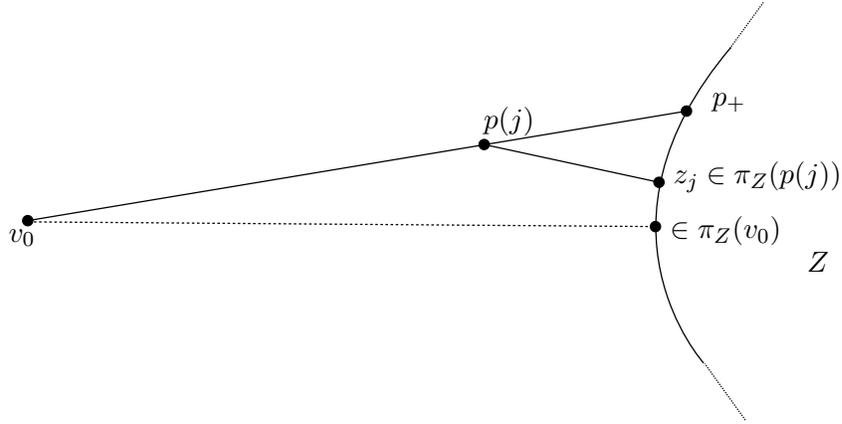
\begin{figure}[ht]
\labellist
\pinlabel $Z$ at 500 100
\pinlabel $\in \pi_{Z}(v_0)$ at 442 119
\pinlabel $z_j\in \pi_{Z}(p(j))$ at 462 153
\pinlabel $p_+$ at 444 200
\pinlabel $v_0$ at 0 115
\pinlabel $p(j)$ at 306 188
\endlabellist
\centering
\includegraphics[scale=0.6]{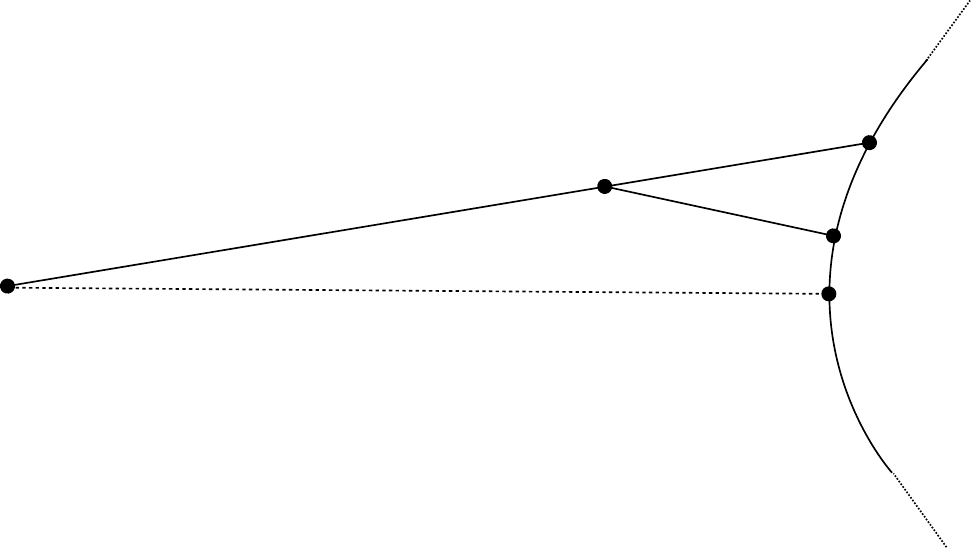}
\label{fig:coset}
\caption{Paths in the proof of Lemma \ref{lem:fftpcoset}.}
\end{figure}
\end{proof}


Since $G$ acts by automorphism on the graph $\ga$, it preserves the distance $\d$, and we see that $\pi_{gZ}(gv)=\pi_{Z}(v)$ for every $g\in G$, $v\in V\ga$. 
It follows that for every $g\in G$,  if $Z$ has $\cte$-fellow projections/$\cte$-bounded projections then so does $gZ$.


Let $H=\gen{\theta(X)}\leqslant G$.
By the $G$-proper embeddings conditions, $\{gZ \mid v_0\in gZ\}$ is a finite set $\{Z_1,\dots, Z_s\}$. Since $H$ is vertex transitive, given $g\in G$, there is $h\in H$ such that $hv_0=gv_0$, and thus $h^{-1}gZ\in \{Z_1,\dots,Z_s\}$. Thus $GZ= \cup_{i=1}^s H Z_i$. Taking a set of $H$-orbit representatives $T\subseteq \{Z_1,\dots, Z_s\}$, we have that  
$\sqcup_{Z_i\in T} HZ_i= GZ$.

We claim that Theorem \ref{thm:mainA+} can be reduced to the case $H=G=\gen{\theta(X)}$. Indeed, we note that
\begin{align*}
&\{w\in X^*\mid p_w \text{ is a $(G,Z)$-geodesic}\}\cap \cL(\cP)\\
= &\cup_{Z_i\in T} \{w\in X^*\mid p_w \text{ is a $(H,Z_i)$-geodesic}\}\cap \cL(\cP)
\end{align*}
and since the finite union of regular languages is regular (see for example \cite[Lemma 1.4.1]{ECHLPT}), the regularity of the language of $(G,Z)$ geodesics follows from the regularity of the language of $(H,Z_i)$-geodesics with $Z_i\in T$.

Similarly, 
\begin{align*}
\sum_{t\geq 0}e_{(G,Z)}(n)t^n & = \sum_{Z_i\in T}\left(\sum_{t\geq 0}e_{(H,Z_i)}(n)t^n \right)\\
\sum_{t\geq 0}e_{(G,Z)}(n)t^n & = \sum_{Z_i\in T}\left(\sum_{t\geq 0}i_{(H,Z_i)}(n)t^n \right)
\end{align*}
and since the sum of finitely many rational functions is a rational function, the raltionality of the functions on the left, follows from the rationality of those in the right of the above equation. 

Thus from now on, we will assume that $G=H=\gen{\theta(X)}$.

\subsection{Regularity of language of \texorpdfstring{$(G,Z)$-}{ }geodesics}
We let $(X, \tau, \phi_0, \cS, \cA)$ denote the \fftp{}-automaton for $(\ga,X,\theta)$ with parameter $\ctee^2$. We will specify the set of accepting states $\cA$  that exactly accept the $(G,Z)$-geodesics.

Suppose that $v_0\in Z$ and that $w\in X^*$ is such that $p_w$ is a geodesic path. 
It might happen that $\ell(p_w)=\d(v_0,wv_0)< \d(v_0,wZ)$ but $p_w$ is a $(G,Z)$-geodesic since there is some $g\in G$ for which  $wv_0\in gZ$ and $\ell(p_w)=\d(v_0,wv_0)=\d(v_0,gZ)$. To deal with this, we will consider the set $\{gZ \mid v_0\in gZ\}$, which by the $G$-proper embedding assumptions, is a finite set $\{Z_1, \dots, Z_n\}$.
Then, $wZ_1,\dots, wZ_n$ is the set of subgraphs  $\{gZ \mid wv_0\in gZ\}$ and thus $p_w$ is a $(G,Z)$-geodesic if and only if there is $i\in \{1,\dots, n\}$ such that $\ell(p_w)=\d(v_0,wZ_i)$.

Recall that all the states of the fftp-automaton are functions $\phi_w$ with $w\in X^*$ that records the distances from $v_0$ to the vertices in $\ball_{wv_0}(\cte)$. We let $\cA$ be the set of those $\phi_w$ that detect that $wv_0$ minimizes the distance from the vertices of $wZ_i$ to $v_0$.
That is
\begin{equation}\label{eq:A}
\cA=\{\phi_w \mid \exists i\in \{1,\dots, n\}\text{ s. t. }\phi_w(u)\geq 0 \;  \forall u\in X^{\leq \cte} \text{ with } wuv_0\in wZ_i\}.
\end{equation}

\begin{lem}\label{lem:ktype of gZ geo}
Let $w\in \cL(P)$. Then $p_w$ is a $(G,Z)$-geodesic if and only if $\phi_w\in \cA$.
\end{lem}
\begin{proof}
If $\phi_w\in \cA$, then $\phi_w\neq \varrho$, and hence $p_w$ is a geodesic. 

The path $p_w$ is a $(G,Z)$-geodesic if and only if there is $g\in G$ such that $(p_w)_+\in gZ$ and $\d(v_0,gZ)=\ell(p)$, and hence, if and only if there is $i$ such that $(p_w)_+\in wZ_i$ and $\d(v_0,gZ)=\ell(p)$. By the Lemma \ref{lem:fftpcoset}, the latter is equivalent to the non-existence of a path $q$, with $\ell(q)<\ell(p_w)$, $q_+\in wZ_i$, $\d((p_w)_+,q_+)\leq \cte$ and such that $q$ and $p_w$ asynchronously $\cte^2$-fellow travel. 
Finally, by equation \eqref{eq:state}, this  is equivalent to  $\phi_w(u)\geq 0$ for all $u\in \{s  \in X^{\leq k} \mid wsv_0\in wZ_i\}$ which is equivalent to $\phi_w\in \cA$. 
\end{proof}
\begin{cor}\label{cor:regularZgeo}
If $\cL(P)$ is regular then the language 
$$\{w\in X^*\mid p_w \text{ is a $(G,Z)$-geodesic}\}\cap \cL(\cP)$$ is regular.
\end{cor}
\begin{proof}
Let $\cL$ be the language accepted by the fftp-automaton with parameter $\cte^2$ and accepting states $\cA$ defined in equation \eqref{eq:A}.
By definition, $\cL$ is a regular language. 
Since the intersection of regular languages is regular, $\cL\cap \cL(P)$ is regular  (see for example \cite[Lemma 1.4.1]{ECHLPT}). 
By Lemma \ref{lem:ktype of gZ geo}, $\cL\cap \cL(P)=\{w\in X^*\mid p_w \text{ is a $(G,Z)$-geodesic}\}\cap \cL(\cP)$.
\end{proof}

\subsection{Rationality of \texorpdfstring{$(G,Z)$-}{ }embeddings}\label{ssec:e(n)}

Within this subsection, we suppose that $Z$ is a finite subgraph of $\ga$. 
By increasing the constant $\cte$, if needed, we can assume that  $\diam(Z)\leq \cte$ and $v_0\in Z$.
We will deduce the rationality of $e_{(G,Z)}$ from the one of $e_{(G,\bullet)}$.
The idea is that the number of $gZ\subseteq \ball_{v_0}(n)$ with $v\in gZ$ is always the same number, say $D$,  except when $v$ is close to the border of the $\ball_{v_0}(n)$ where copies of $Z$ might be not embeddable and we need a finer count. 
We will see that this number can be read from $n-\dist(v_0,v)$ and the $\cte$-type  function defined by 
$\kappa_v(u)=\dist(v_0,u)-\dist(v_0,v)$ for $u\in \ball_{v}(\cte)$.
For $i=0,1,\dots, \ctee$, let
\begin{align*}
\cS_v^i&=\{gZ \mid \exists h\in G \text{ s.t. } hv_0=v,\, hZ=gZ,\, \kappa_v(gZ)\subseteq [-\ctee, i] \}\\
&=\{gZ \mid \exists h\in G \text{ s.t. } hv_0=v,\, hZ=gZ\subseteq \ball_{v_0}(\dist(v_0,v)+i)\}.
\end{align*}
and let $C_v^i=\sharp \cS_v^i$.

We will show now that if $u$ and $v$ have the same $\cte$-type mod $G$, then $C_v^i=C_u^i$. Recall that if $u,v$ have the same $\cte$-type mod $G$ then there is $k\in G$ such that $kv=u$ and $\kappa_v(z)=\kappa_u(kz)$ for all $z\in \ball_v(\cte)$.  Observe that the map $\cS_v^i\to \cS_u^i$ given by $gZ\mapsto kgZ$, is a bijection. Indeed, if $g\in G$ with $gv_0=v$ and $\kappa_v(gZ)\leq i$ then $kgv_0=u$ and for each $z\in gZ$, $\kappa_u(kz)=\kappa_v(z)\leq i$ and thus $\kappa_u(kgZ)\leq i$.
This shows that the map $\cS_v^i\to \cS_u^i$ is well defined, and since it is a restriction from a group action, it is injective. One easily checks that multiplying elements of $\cS_u^i$ by $k^{-1}$ gives the inverse map and hence we have a bijection.  
Thus, if $u\sim_\cte v$  then $C_v^i=C_u^i$.

Note that $C_v^\ctee$ is independent of $v$ and we call this number just by $D$.


Let $g,h\in G$ and suppose that $gZ=hZ$ with $gv_0\neq hv_0$. 
Then there is automorphism of $Z$ that sends $v_0$ to $g^{-1}hv_0$. 
Let $O=G_Z v_0 $ be the orbit of $v_0$ under  $G_Z=\{g\in G \mid gZ=Z\}$, the $G$-stabilizer of $Z$. 
Then, we have that
\begin{align*}
e_{(G,Z)}(n) & \coloneq \sharp\{gZ \mid gZ\subseteq \ball_{v_0}(n)\}\\
&=\dfrac{1}{|O|}\sum_{v\in \ball_{v_0}(n)}\sharp\{gZ \mid \exists h\in G \text{ s.t. }  hv_0=v,\, hZ=gZ\subseteq \ball_{v_0}(n)\}.
\end{align*}

We split the value of $|O| e_{(G,Z)(n)}$ in terms of $\cte$-type of the vertices of $\ball(v_{0})$ and their distance to $v_0$ as follows
\begin{align*}
|O| e_{(G,Z)}(n)&=\sum_{v\in \ball_{v_0}(n-\cte)}\sharp \cS_v^{\cte} +\sum_{i=0}^{\cte-1}\sum_{\sigma \in V/\sim_\cte}\sum_{\stackrel{v \text{ of type $\sigma$}}
{ \d(v_0,v)=n-i}} \sharp \cS_v^{i}\\
& = D \cdot e_{(G,\bullet)}(n-\cte)+\sum_{i=1}^\ctee\sum_{\sigma\in V/\sim_\ctee} C_\sigma^i \cdot \ol{e}^\sigma_{(G,\bullet)}(n-i)
\end{align*}
where $\ol{e}_{(G,\bullet)}^{\sigma}$ is  defined in Definition \ref{defn:auxilariy e}. We recall that the functions $\kappa_v$, and  hence the constants $D$ and $C_\sigma^i$ from the states $\phi_w$ of the fftp-automaton. 

We have expressed $e_{(G,Z)}(n)$ as finitely many sums of the $e_{(G,\bullet)}(n)$ and $\ol{e}_{(G,\bullet)}^\sigma(n)$. It follows  by Theorem \ref{thm:Markov} that $\sum_{n\geq 0} e_{(G,Z)}(n)t^n$ is a sum of finitely many rational functions.

\subsection{Rationality of \texorpdfstring{$(G,Z)$-}{ }intersections}\label{ssec:i(n)}
In this subsection, we no longer assume that $Z$ is finite, but we assume  that  $Z$ has  $\cte$-bounded projections.
As in the previous subsection, the idea is to express
$$i_{(G,Z)}(n)=\sharp\{gZ \mid gZ \cap \ball_{v_0} \neq \emptyset\}$$
in terms of $e_{(G,\bullet)}$ and $e^\sigma_{(G,\bullet)}$. 

For each $v\in V\ga$ consider
$$\cS_v=\sharp\{gZ \mid v\in gZ, \dist(v_0,v)=\dist(v_0,gZ)\}.$$
Since $Z$ has $G$-proper embeddings $\cS_v$ is a finite set. We will see if $u$ and $v$ are in the same $\cte$-type mod $G$, then there is an automorphism of $G$ inducing a bijection between $\cS_v$ and $\cS_u$.  Recall that if $u,v$ have the same $\cte$-type mod $G$ then there is $h\in G$ such that $hv=u$ and $\kappa_v(z)=\kappa_u(hz)$ for all $z\in \ball_v(\cte)$. We claim that the map
$\cS_v\to \cS_u$
given by $gZ\mapsto hgZ$ is a bijection.
This follows easily from Lemma \ref{lem:fftpcoset}, which implies that if $v\in gZ$, then  $\dist(v_0,gZ)=\dist(v_0,v)$ if and only if $\kappa_v(z)\geq 0$ for all $z\in \ball_{v}(\cte)\cap gZ$.



One cannot use $\sum_{v\in V/\sim_\cte} \sharp\cS_{v} \cdot e^{[v]}_{(G,\bullet)}$ to compute $i_{(G,Z)}$ since that might  produce some overcounting.
Given $gZ$, recall that $v_i\in \pi_{gZ}(v_0)$ if and only if
$$v_i\in gZ \text{ and } \dist(v_0,v_i)=\dist(v_0,gZ).$$
By the $\cte$-bounded projections assumptions $\diam(\pi_{gZ}(v_0))\leq \cte$. Therefore, we can read the whole $\pi_{gZ}(v_0)$ from $\kappa_{v}$ with $v\in \pi_{gZ}(v_0)$ since $$\pi_{gZ}(v_0)= \{z\in gZ \mid \kappa_{v_i}(z)=0\}.$$ 
We define 
$$C_{v}= \sum_{gZ\in \cS_v} \frac{1}{\sharp \{z\in gZ \mid \kappa_{v_i}(z)=0\}},$$
and as we have seen, $C_{v}$ only depends on the $\cte$-type of $v$.

We have 
\begin{align*}
\sum_{n\geq 0} i_{(G,Z)}(n) t^n & = \sum_{n\geq 0}\sum_{\sigma \in V\ga/\sim_\cte} {C_\sigma}\,e^{\sigma}_{(G,\bullet)}(n)t^n.
\end{align*}
By Theorem \ref{thm:Markov}, the right-hand side of the above expression (and hence the left-hand side) is a rational function.

\subsection{Proof of Theorem \ref{thm:mainA+}}
\begin{proof}[Proof of Theorem \ref{thm:mainA+}] Item (1) follows from Corollary \ref{cor:regularZgeo} and the fact that the growth series of a regular language is rational. Item (2) follows from the discussion of Subsection \ref{ssec:e(n)} and Item (3) follows from the discussion of Subsection \ref{ssec:i(n)}.
\end{proof}

\section{Schreier coset graphs}

Suppose that $X$ is a symmetric generating set of a group $G$, i.e.  $X=X^{-1}$, then the involution on $X$, extends to an involution on $X^*$, also denoted by $^{-1}$, defined by $x_1x_2\dots x_n\mapsto x_n^{-1}x_{n-1}^{-1}\dots x_1^{-1}$. 
Observe that  $\geol(G/H,X)=\geol(H\backslash G, X)^{-1}$, where $\geol(G/H,X)=\{w\in X^* \mid \ell(w)\leq \ell(u) \, \forall u\in X^*, u\in wH\}$ and $\geol(H\backslash G,X)$ was defined similarly in Theorem \ref{thm:mainB}.
In particular, since the reverse of a regular language is regular \cite[Theorem 1.2.8]{ECHLPT}, we have that $\geol(G/H, X)$ is regular if and only if $\geol(H\backslash G, X)$ is regular.

Let $|gH|_X=\min\{ \ell(w) \mid w\in X^*,\, w\in gH\}$ and analogously $|Hg|_X=\min\{ \ell(w) \mid w\in X^*,\, w\in_G Hg\}$.
Note that $|Hg^{-1}|_X=|gH|_X$ and hence,   for each $n\in \mathbb{N}$ $|\ball_{(H\backslash G,X)}(n)|=|\ball_{(G/H,X)}(n)|$ where
$\ball_{(G/H,X)}(n)=\{gH\in G/H\mid |gH|_X\leq n\}.$

%

One advantage of working with left cosets as subsets of $\ga=\ga(G,X)$  is that the left $G$-action  on $G/H$ is isometric, i.e. $\d_\ga(aH,bH)=\d_\ga(caH,cbH)$ for all $a,b,c\in G$.
 
As noticed before, $H$ has $D$-fellow projections if and only if $gH$ has $D$-fellow projections.

\begin{rem}\label{rem:notationSchreier}
Remind Remark \ref{rem: presentation Cayley}. In our situation, the labeling of paths of the Cayley graph $\ga(G,X)$ coincides with the canonical presentation of paths starting at $v_0=1_G$.
If $Z$ is a subgraph of $\ga$, then we will denote the set of $(G,Z)$-geodesics simply by $\geol(G/Z, X)$. 
We note that this agrees with the notation above. 
\end{rem}

\begin{cor}\label{cor:Schreierfftp}
If $(G,X)$ has \fftp{} and $H\leqslant G$ has fellow projections in $(G,X)$, then $\ga(G,H,X)$ has \fftp{} relative to the collection of paths starting at $H$.
\end{cor}
\begin{proof}
Let $\cte$ be the \fftp{} constant for $(G,X)$ and the fellow projections constant for $H$ in $(G,X)$.
We will see that $\ga(G,H,X)$ has $\cte^2$-\fftp{}.

Let $p$ be a path in $\ga(G,H,X)$ that is not geodesic.
Let $w\in X^*$ be the label of $p$.
Then the path $p_{w^{-1}}$ in $\ga(G,X)$ from $1$ to $w^{-1}H$ is not an $H$-geodesic. 
By Lemma \ref{lem:fftpcoset}, since $Z=w^{-1}H$ has $\cte$-fellow projections, there is $u\in X^*$ such that $\ell(u)<\ell(w)$, $uH=w^{-1}H$, and $p_u$ and $p_{w^{-1}}$
asynchronously $\cte^2$-fellow travel and $d_X((p_u)_+,(p_{w^{-1}})_+)=|u^{-1}w^{-1}|_X\leq \cte$.
Take $h\in H$ such that $uh=w^{-1}$.
Then the path $p_w$ in $\ga(G,X)$ starting at 1 and labelled by $w$, $\cte^2$-fellow travels with the path $h^{-1}p_{u^{-1}}$ starting at $h^{-1}$ and labelled by $u^{-1}$.
Note that the there is a graph map $\rho\colon \ga(G,X)\to \ga(G,H,X)$ that preserves labels and does not increase distances.
Thus, $\rho(h^{-1}p_{u^{-1}})$ and $\rho(p_{w})$ have the same endpoints, and asynchronously $\ctee^2$-fellow travel.
\end{proof}

\begin{proof}[Proof of Theorem \ref{thm:mainB}]
Item (1) is Corollary \ref{cor:Schreierfftp}. By Theorem \ref{thm:mainA+} (1), $\{w\in X^* \mid p_w 
\text{ is a $(G,Z)$-geodesic}\}\cap\cL(\cP)$ is regular (since $\ga(G,X)$ has fftp relative to $\cP$ the set of all the paths $\cP$ starting at $H$, $\cL(\cP)=X^*$) and it is equal to $\geol(G\backslash H,X)$, and by the discussion above $\geol(H/G,X)$ is also regular. 
Finally, (3) follows  from Theorem \ref{thm:mainA+}~(3).
\end{proof}

\subsection{Shortlex coset transversals}

Fix an order on $X$ and denote by $\leq_{\mathsf{SL}}$ the shortlex order on $X^*$, i.e. $w\leq_{\mathsf{SL}} v$ if and only if $\ell(w)<\ell(v)$ or $\ell(w)=\ell(v)$ and $w$ precedes $v$  lexicographically (with the fixed order on $X$). Let $\mathsf{ShortLex}(G/H,X)=\{w\in X^* \mid w\leq_{\mathsf{SL}}v, \,\forall v\in X^* \text{ with } vH=wH\}$ and similarly $\mathsf{ShortLex}(H\backslash G,X)=\{w\in X^* \mid w\leq_{\mathsf{SL}}v, \, \forall v\in X^* \text{ with } Hv=Hw\}$.

\begin{prop}\label{prop:shortlexcosets}
Suppose that $(G,X)$  is shortlex automatic and $H\leqslant G$ is a subgroup with  bounded projections. Then $\mathsf{ShortLex}(G/H,X)$ is a regular language.
\end{prop}
\begin{proof}
Without loss of generalization, we can assume that $\ctee$ is the fftp constant and the bounded projections constant.
Recall that if $(G,X)$ is shortlex automatic it means that  $\mathsf{Shortlex}(G,X)$ is a geodesic automatic structure.
Let $\cP$ be the paths in $\ga(G,X)$ with label in $\mathsf{Shortlex}(G,X)$. By the Example \ref{ex:relativefftp} $\ga(G,X)$ has fftp relative to $\cP$ and $\cP$ is $1_G$-spanning. 
%

By Corollary \ref{cor:regularZgeo}, we have that $\cL=\geol(G/H,X)\cap \mathsf{Shortlex}(G,X)(\{\varepsilon\}\cup X)$ is a regular language. It is clear that $\mathsf{ShortLex}(G/H,X)\subseteq \cL$. Now suppose that there are $w, w'\in \cL$ with $wH=w'H$. By the bounded projection property $\d_X(w,w')\leq M$ and since $w,w'\in \mathsf{ShortLex}(G,X)$, $w$ and $w'$ $M^2$-fellow travel.

In particular $$\mathsf{Shortlex}(G/H,X)=\{w\in \cL \mid w\leq_{\mathsf{SL}}w', \,\forall w'\in \cL, w'H=wH\}$$
and bearing in mind that asynchronous fellow travel of geodesics imply synchronous fellow travel (with potentially a larger constant), there is an standard argument showing that right-hand set is a regular language (for example \cite[Proof of Theorem 2.5.1]{ECHLPT}).
\end{proof}

\begin{ex}[Redfern] Suppose that $G$ is hyperbolic, $X$ is an ordered generating set and $H$ is a quasi-convex subgroup (and as discussed in the next section, it has bounded projections). Hyperbolic groups are fftp and shortlex automatic, and hence it follows that $\mathsf{ShortLex}(G/H,X)$ is regular.  Redfern in his PhD thesis \cite{Redfern} proved that $(G,H,X)$ is shortlex coset automatic, which implies  that $\mathsf{ShortLex}(H\backslash G,X)$ is regular. 
\end{ex}

\section{Examples}\label{sec:ex}
The falsification by fellow traveler property is known to depend on the generating set \cite{NS} (see also \cite{Elvey-Price}). For some families of fftp graphs, we will provide examples of subgraphs with bounded projections. Recall that a subgraph $Z$ of $\ga$ is {\it $\sigma$-quasi-convex}, if for every $x,y\in Z$ every geodesic path $p$ from $x$ to $y$ lies in the $\sigma$-neighbourhood of $Z$.


\subsection{Quasi-convex subsets of hyperbolic spaces}
It follows immediately  from the thin-triangle definition, that if a graph is $\delta$-hyperbolic, then it has $\delta$-fftp. The following follows from \cite[Corollary 2.3]{CDP}.

\begin{lem}\label{lem:hyper->boundedproj}
Let $\ga$ be a $\delta$-hyperbolic graph, and $Z$ a $\sigma$-quasi-convex subset. 
There exists a constant $k$, depending only of $\delta$ and $\sigma$ such that for any $x,y\in Z$ and any $z_x\in \pi_Z(x)$, $z_y\in \pi_Z(y)$ one has that $\dist(z_x,z_y)\leq k+ \dist(x,y)$.
\end{lem}

\begin{cor} 
Quasi-convex subgroups of hyperbolic groups have bounded projections.
\end{cor}

As we have seen in the proof of Theorem \ref{thm:mainA+}, the fftp constant and the bounded projection constant determine the size of the digraph codifying the random Markov geodesic combing. 
We will use now the idea of \cite[Lemma 8.2]{EpsteinGrowth}, to see that in the case of $\delta$-hyperbolic graphs the construction of the Markov geodesic combing depends mainly on $\delta$ and no other parameter.

The idea is that if $Z$ is a subset of $\ga$ with $D$-bounded projections and $v\in V\ga$ is a vertex far away from $Z$, then $Z$-geodesics from $v$ $\delta$-fellow travel except from a constant time depending on $D$. 

We will use the following well known fact.
\begin{lem}\label{lem:triangle Epstein}
Let $p$, $q$ be geodesics in a $\delta$-hyperbolic spaces with $p_-=q_-$. Then there is $M=M(\delta)$ and $R=R(\dist(p_+,q_+),\delta)$ such that the initial subpaths $p'$ and $q'$ of $p$ and $q$ of length $\min \{\ell(p),\ell(q)\}-\cteee$ synchronously $\ctee$-fellow travel.
\end{lem}
\begin{proof}
 Let $t$ be a geodesic from $p_+$ to $q_+$. Then $p,q,r$ form a geodesic triangle, and as we are in hyperbolic space, the triangle has a $\delta'$-center (where $\delta'$ only depends on $\delta$), that is, there is a point $x$ that is at distance $\delta'$ of the other three sides. 
Say that $x_p\in p$, $x_q\in q$ and $x_t\in t$ are three vertices at distance at most $2\delta'$ of each other. It follows that $x_p$ (resp. $x_q$) is at distance at most $R=\dist(p_+,q_+)+2\delta'$ of $p_+$ (resp. $x_q$). In particular the subpath of $p$ from $p_-$ to $x_p$ and the subpath of $q$ from $q_-$ to $x_q$ $M$-synchronously fellow travel with a constant depending only on $\delta$ (one can take $M=2\delta'\cdot \delta$).
\end{proof}

%

\begin{proof}[Proof of Theorem \ref{thm:mainC}]
We now assume that $G$ is an hyperbolic group, $X$ is a finite symmetric  generating set, $\ga=\ga(G,X)$ the Cayley graph of $G$ and $Z$ is a quasi-convex subgroup of $G$ with $D$-bounded projections. Let $\ctee=M(\delta)$ and $\cteee(D,\delta)$ the constants of the previous lemma. Without loss of generality we can assume that $D<\ctee$.

We will show that there is a polynomial $Q_X(t)\in \Z[t]$, only depending on $G$ and $X$,  such that $Q_X(t)\cdot (\sum_{n\geq 0} i_{(G,Z)}(n)t^n)$ is a polynomial. 

Let $\ol{i}_{(G,Z)}(n)=i_{(G,Z)}(n)-i_{(G,Z)}(n-1)$ for $n\geq 0$ and setting $i_{(G,Z)}(-1)=0$. Observe that $$(1-t)\sum_{n\geq 0} i_{(G,Z)}(n)t^n=\sum_{n\geq 0}\ol{i}_{G,Z}(n)t.$$
Thus, it is enough to show that $Q_X(t)(\sum_{n\geq 0} \ol{i}_{(G,Z)}(n)t^n)$ is a polynomial

Let $n> \cteee$ and 
$$\cS(n)=\{gZ \mid gZ\cap \ball_{1_G}(n)\neq \emptyset =gZ \cap \ball_{1_G}(n-1)\}.$$
Note that $\ol{i}_{(G,Z)}(n)$ is the cardinality of  $\cS(n)$.
Let $v\in \ga$, with $\d(1_G,v)=n-\cteee$ and let 
$$\cS_v(n)=\{gZ\in \cS(n)\mid \exists p \text{ $gZ$-geodesic  from }v_0 \text{ passing throug } v\}.$$
We claim that $\sharp \cS_v(n)$ only depends on the $2\cte$-type of $v$ mod $G$. 
Indeed, suppose that $u\sim_\cte v$, then there is $h\in G$ and such that $hv=u$ and $\kappa_v(z)=\kappa_u(hv)$ for all $z\in \ball_v(2\ctee)$. Let $gZ\in \cS_v(n)$, we will show that $hgZ\in \cS_u(n)$.  
The argument is very similar to the one of Lemma \ref{lem:cone_types}. 
Let $p$ be a $gZ$-geodesic from $v_0$ to $gZ$ passing through $v$. 
We can divide $p$ as $p=p_1p_2$ where $(p_1)_+=v=(p_2)_-$. 
Observe that $q_2=h p_2$ is a geodesic path of length $R$ from $hv=u$ to $hgZ$. 
Let $q_1$ be any geodesic from $v_0$ to $u$. 
We claim that $q=q_1q_2$ is a $hgZ$-geodesic. In fact, we already know by Lemma \ref{lem:cone_types} that $q$ is geodesic. 
If $q$ is not $hgZ$ geodesic, by Lemma \ref{lem:fftpcoset}, there is a path $r$ from $v_0$ to $hgZ$, shorter than $q$  with $\dist(q_+,r_+)\leq \ctee$. Without loss of genarlity, we can assume that $r$ is geodesic.   
Let $r=r_1r_2$ with $r_2$ a subpath of lenth $\cteee$. 
Note that $\ell(q)>\ell(r)\geq \ell(q)-M$. Thus $\min\{\ell(q),\ell(r)\}-\ctee\geq \ell(q_1)-\ctee$ and we get that $\dist((q_1)_+,(r_1)_+)\leq 2\ctee$, and therefore, $(r_1)_+=(r_2)_-\in\ball_u(2\ctee)$. 
Then 
$$\kappa_u((r_1)_+)=\dist(v_0,(r_1)_+)-\dist(v_0,u)=\dist(v_0,(r_1)_+)+R-(\dist(v_0,u)+R)\leq \ell(r)-\ell(q)<0.$$
Now consider any geodesic path $p'_1$ from $v_0$ to $h^{-1}(r_1)_-$ and $p'_2=h^{-1}r_2$. Let $p'=p_1'p_2'$ and note that $p'_+=h^{-1}(r_2)_+\in gZ$. Moreover, 
$$\ell(p')-\ell(p)=\ell(p_1')-\ell(p_1)+\ell(p_2')-\ell(p_2)=\ell(p_1')-\ell(p_1)=\kappa_v((p_1)'_+)=\kappa_u((r_1)_+)<0$$
and thus $p$ is not a $gZ$-geodesic and we get a contradiction. Thus q is a $hgZ$-geodesic and hence, since $\ell(q)=\ell(p)=n$, $hgZ\in \cS_u(n)$.  Similarly, if $gZ\in \cS_u(n)$ then $h^{-1}gZ\in \cS_v(n)$ and thus $\cS_v(n)$ and $\cS_u(n)$ have the same cardinality.

Given $gZ\in \cS_v(n)$, there might be $v_1,\dots, v_s\in V\ga$ such that $gZ\in \cS_{v_i}(n)$. We call this vertices the {\it $\cteee$-parents} of $gZ$. Note that by Lemma \ref{lem:triangle Epstein}, $\dist(v_i,v)\leq M$ for $i=1,\dots, s$.  Again, we claim that the number of $\cteee$-parents of $gZ\in \cS_v(n)$ only depends on the  $2\ctee$-type if $v$. Indeed, there are paths $r_1,\dots r_s$ of length $\cteee$ with $(r_i)_-=v_i$ and $(r_i)_+\in gZ$ with $\kappa_v(v_i)=0.$ Now if $u=hv$ and $\kappa_v(z)=\kappa_u(hz)$ for all $z\in \ball_{v}(2\ctee)$, it follows easily that $hgZ\in \cS_{hv_i}(\dist(v_0,u)+\cteee)$ for $i=1,\dots,s$. So the number of $\cteee$-parents of $gZ\in \cS_v(\dist(v_0,v)+\cteee)$ is at least the number of $\cteee$-parents of $hgZ\in\cS_u(\dist(v_0,u)+\cteee)$. A symmetric argument shows that indeed the numbers are equal.

Let $$C_v(\cteee)=\sum_{gZ\in \cS_v(\dist(v_0,v)+\cteee)}\dfrac{1}{\sharp\{\text{$\cteee$-parents of gZ}\}}.$$
Now we have that
$$\sum_{n\geq \cteee}\ol{i}_{(G,Z)}(n)t^n =\sum_{n\geq \cteee}\sum_{\sigma \in V\ga/\sim_{2\ctee} } C_v(\cteee)\cdot \ol{e}^{\sigma}_{(G,\bullet)}(n-\cteee) t^n.$$

By Theorem \ref{thm:Markov}, there are  polynomials $P_\sigma$ and $Q_\sigma$ over $\mathbb{Z}[t]$ such that 
$P_{\sigma}(t)/Q_\sigma(t)=\sum_{n\geq 0}\ol{e}^{\sigma}_{(G,\bullet)}(n)t^n$
and thus 
$\dfrac{P_{\sigma}(t)t^\cteee }{Q_{\sigma}(t)}=\sum_{n\geq \cteee}\ol{e}^{\sigma}_{(G,\bullet)}(n-\cteee) t^n$. 
Therefore 
$$\sum_{n\geq 0}\ol{i}_{(G,Z)}t^n=\sum_{i=0}^{\cteee-1}\ol{i}_{(G,Z)}t^n + \sum_{\sigma \in V\ga/\sim_{2\ctee}} C_\sigma(R) \dfrac{P_{\sigma}(t)t^\cteee }{Q_{\sigma}(t)} $$ and we can take $Q_X(t)=\prod_{\sigma} Q_\sigma (t)$. This completes the proof of (1).

To show (2), it follows from the previous computation, Theorem \ref{thm:Markov} and  \cite[Proposition 3.5.]{WiseFuter} that $\limsup_{n\to \infty}\sqrt[n]{i_{(G,Z)}(n)}=\max\{\rho_{\mathbf{A}},1\}$ where $\mathbf{A}$ is the matrix provided by Theorem \ref{thm:Markov} and $\rho_{\mathbf{A}}$ is the Perron-Frobenius eigenvalue. If $G$ is non-elementary and $H$ is of infinite index, then $\ga(G,H,X)$ grows exponentially (Kapovich \cite[Theorem 1.5]{Kapovich}  shows there is a quasiconvex free subgroup $F\leq G$ of rank two such that $H^g\cap F=\{1\}$ for every $g\in G$) and hence $\rho_{\mathbf{A}}>1$ and thus we have that $\limsup_{n\to \infty}\sqrt[n]{i_{(G,H)}(n)}=\rho_{\mathbf{A}}>1$.

Recall that Koubi \cite{Koubi} showed that if $G$ is non-elementary, then there is $\lambda>1$ such that for any finite generating $Y$ set of $G$ it holds that $\limsup_{n \to \infty}\sqrt[n]{|\ball_Y(n)|}>\lambda$. Thus, we get that $\rho_{\mathbf{A}}>\lambda$ for all generating sets $X$.
%
%
%
\end{proof}

\subsection{Parabolic subgroups of relatively hyperbolic groups}
Let $G$ be a group hyperbolic relative to a collection of subgroups $\{H_\omega\}_{\omega \in \Omega}$. The subgroups $H_\omega$, $\omega\in \Omega$ are called {\it parabolic subgroups} (see \cite{Osin06} for the details). Let $\cH=\cup H_\omega$. 

In this subsection we will make intense use of results of \cite{AC3} where it is shown that  fftp is preserved under relative hyperbolicity in the following way. 
Suppose that $Y$ is a finite generating set for $G$.
 There is a finite subset $\cH'\subseteq \cH$ such that if $X$ is a finite generating satisfying that $Y\cup \cH \subseteq X\subseteq Y\cup \cH$ and that $\ga(H_\omega, X\cap H_\omega)$ is fftp, then $\ga(G,X)$ is fftp. We remark that the condition $\ga(H_\omega, X\cap H_\omega)$ is fftp, it is relatively easy to achieve, since if $H_\omega$ has fftp for some generating set, then any finite generating set of $H_\omega$ can be enlarged to have fftp (see \cite[Proposition 3.2]{AC3}).

Let $p$ be a path in the Cayley graph $\ga(G,X\cup \cH)$. An {\it $H_\lambda$-component} of $p$, is a subpath $s$ of $p$ with the property that the label of the path $s$ is an element of the free monoid $H_\lambda^*$ and it is not properly contained in any other subpath of $p$ with this property.
 Two components ${s}$ and ${r}$ (not necessarily in the same path) are {\it connected} if both are $H_\omega$-components for
some $\omega\in \Omega$ and $({s}_-)H_\omega=({r}_-)H_\omega$. A component in a closed path that is not connected to other component is called {\it isolated}.

We will use the following result, which is a version of \cite[Proposition 3.2]{OsinPF}.
\begin{lem}\label{lem:gon}
Let $G$ be hyperbolic relative to $\{H_\omega\}_{\omega \in \Omega}$ and $X$ a finite generating set of $G$. 
There exists $D=D(G,X,\lambda,c)>0$ such that the following hold.
Let $\mathcal{P}=p_1p_2\cdots p_n$ be an $n$-gon in $\ga(G,X\cup \cH)$ and  $I$ a distinguished subset  of sides of $\mathcal{P}$ such that if $p_i\in I$, $p_i$ is an isolated component in $\mathcal{P}$, and
if $p_i\notin I$, $p_i$ is a $(\lambda,c)$-quasi-geodesic. Then
$$\sum_{i\in I}\d_X((p_i)_-,(p_i)_+)\leq D n.$$
\end{lem}
\begin{lem}\label{lem:bp_parabolic}
Let $G$ be  hyperbolic relative to $\{H_\omega\}_{\omega\in \Omega}$ and $Y$ a finite symmetric generating set. 
Then there is a finite subset $\cH'\subseteq \cH$ such that for every finite generating set $X$ of $G$ satisfying that $Y\cup \cH'\subseteq X\subseteq Y\cup \cH$,  any  $H_\omega$ has bounded projections (and hence fellow projections) in $\ga(G,X)$.
\end{lem}
\begin{proof}
From  the Generating Set Lemma \cite[Lemma 5.3]{AC3}, 
there is a finite set $\cH'$ of $\cH$ and constants $\lambda\geq 1, c\geq 0$ with the property that
for any finite symmetric generating set $X$ such that $Y\cup \cH'\subseteq X\subseteq Y\cup \cH$
one has that for any geodesic word $w\in X^*$ there is a word $\wh{w}$ in $(X\cup \cH)^*$
with $w=_G\wh{w}$, $\wh{w}$ labels a $(\lambda,c)$-quasi-geodesic in $\ga(G,X\cup \cH)$  
and such that any prefix of $\wh{w}$ represents a prefix of $w$ in the following way: 
if $w\equiv x_1\dots x_n$ and $\wh{w}\equiv z_1\dots z_m$, $m\leq n$, 
there is an increasing function $f\colon \{1,\dots, m\}\to \{1,\dots,n\}$ such that $z_1\dots z_{i}=_G x_1\dots x_{f(i)}$. In particular, the set of group elements that appears as vertices in the path in $\ga(G, X\cup \cH)$ starting at $g$ and labeled by $\wh{w}$ is a subset of the group elements that appears as vertices in the path in $\ga(G,X)$ starting at $g$ and labeled by $w$.

Fix a symmetric generating set $X$ as above. Fix $\omega\in \Omega$.  
We will use $\pi^X$ to denote projections in the space $\ga(G,X)$ and $\pi^{X\cup \cH}$ to denote projections in $\ga(G,X\cup \cH)$. 
Let $a,b\in G$, with $\d_X(a,b)=1$. 
Let $z_a\in \pi^X_{H_\omega}
(a)$ and $z_b\in \pi^X_{H_\omega}(b)$ and let $\wh{z_a}\in \pi_{H_\omega}^{X\cup \cH}$ and $\wh{z_b}\in \pi_{H_\omega}^{X\cup \cH}(b)$.

We will first show that $\d_X(\wh{z_a},\wh{z_b})$ is uniformly bounded, and then that $\d_X(z_a,\wh{z_a})$ and $\d_X(z_b,\wh{z_b})$ are uniformly bounded.

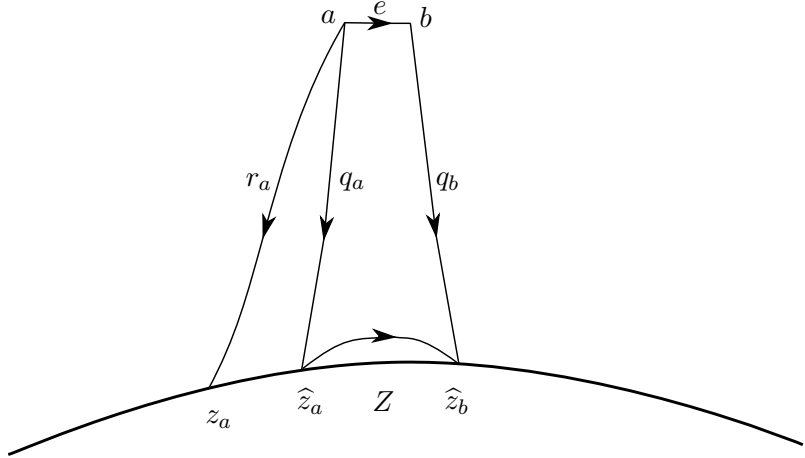
\begin{figure}[ht]
\labellist
\pinlabel $Z$ at 205 30
\pinlabel $\wh{z}_a$ at 165 30
\pinlabel $\wh{z}_b$ at 245 30
\pinlabel $a$ at 175 240
\pinlabel $b$ at 228 240
\pinlabel $e$ at 203 245
\pinlabel $q_a$ at 188 150
\pinlabel $q_b$ at 240 150
\pinlabel $r_a$ at 137 150
\pinlabel $z_a$ at 115 20
\endlabellist
\centering
\includegraphics[scale=0.7]{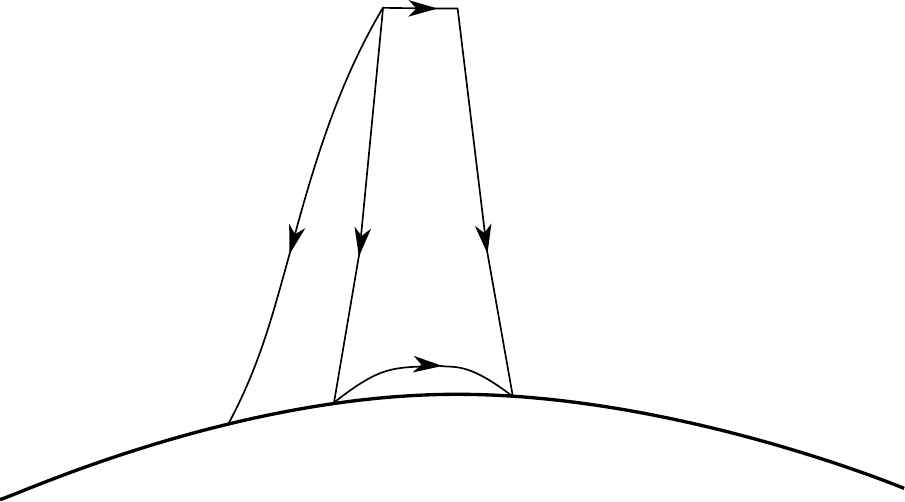}
\caption{Paths involved in the proof of \ref{lem:bp_parabolic} }
\end{figure}

Let $D$ be the constant of Lemma \ref{lem:gon}. 
An important fact is that there is a constant $m>0$ such that if an $H_\omega$-component $s$ is connected to an $H_\mu$-component, and $\d_X(s_-,s_+)>m$, then $\mu=\omega$ (see for example, \cite[Lemma 4.2]{AC3}). 
Without loss of generality, we can assume $m<D.$

Let $q_a$ and $q_b$ be geodesics in $\ga(G,X\cup \cH)$ from $a$ to $\wh{z_a}$ and from $b$ to $\wh{z_b}$ respectively. 
Let $e$ be the edge with $e_-=a$, $e_+=b$ (this edge has its label in $X$). 
Let $f$ be the edge with $f_-=\wh{z_a}$ and $f_+=\wh{z_b}$ (and label in $H_\omega$) and consider the geodesic $4$-gon with sides $e,f,q_a,q_b$. 
We have to bound $\d_X(f_-,f_+)$. 
Suppose that $\d_X(f_-,f_+)>m$. 
If $f$ is not isolated in the $4$-gon, say it is connected to a component of $q_a$, then this component must be an $H_\omega$-component and hence $\d_{X\cup \cH}(a,\wh{z_a})>\d_{X\cup \cH}(a, H_\omega)$ getting a contradiction. 
We obtain a similar contradiction if  $f$ is connected to a component of $q_b$. 
If $e$ and $f$ are connected, then $e$ is an $H_\omega$-component and  $\{a,b\}\in H_\omega$ and then $a=\wh{z_a}$ and $b=\wh{z_b}$ and $\d_X(z_a,z_b)=1$. 
Thus, we are left with the case where $f$ is isolated, and hence  by  Lemma \ref{lem:gon}, $\d_{X}(\wh{z_a},\wh{z_b})\leq 3D$.

Now let $w\in X^*$ a label of a geodesic path from $a$ to $z_a$. 
Let $\wh{w}$ be as above, and $r_a$ the paths from $a$ to $z_a$ in $\ga(G,X\cup \cH)$ with label $\wh{w}$.
Let $t$ be the edge from $z_a$ to $\wh{z_a}$ (with label in $H_\omega$).  
Consider the $(\lambda,c)$-quasi-geodesic triangle with sides $r_a, q_a, t$.  
Assume that $d_X(t_-,t_+)>m$. 
By the above argument, $t$ cannot be connected to a component of $q_a$. 
If $t$ is connected to a component of $r_a$, then there is vertex of $r_a$ in $H_\omega$, which means that there is a prefix of $w_0$ of $w$ such that $aw_0\in H_\omega$, which means that $\d_X(a,p_a)>\ell(w_0)\geq \d_X(a,H_\omega)$ giving a contradiction. 
Then  $t$ is isolated in the triangle and $\d_X(p_a,\wh{z_a})\leq 2D$.

The same arguments shows that $\d_X(z_b,\wh{z_b})\leq 2D$ and hence $\d_X(z_a,z_b)\leq 7D$.
\end{proof}

Combining the lemma with the results of \cite{AC3} discussed in the introduction we have:
\begin{cor}
Let  $G$ finitely generated and hyperbolic relative to $\{H_\omega\}_{\omega \in \Omega}$ such that each $H_\omega$ has the falsification by fellow traveler property respect to some finite generating set. Then there is a finite generating set $X$ of $G$ such that $\ga(G,X)$ is fftp and each $H_\omega$ has bounded projections in $\ga(G,X)$.
\end{cor}

\subsection{Quasi-convex subsets of CAT(0) cube complexes}
Let $\cC$ be a locally finite, CAT(0) cube complex. 
Let $\ga$ be the $1$-skeleton of $\cC$. Then $\ga$ is an fftp graph.
This fact appears in the proof  \cite[Theorem 1.1.]{NoskovCC}, where there is an extra hypothesis to make $\ga$ a Cayley graph, however, the fact that it is a Cayley graph is not used through the proof (for showing fftp), the point being that a path in $\ga$ is geodesic if and only if it does not cross twice the same wall. Using the CAT(0)-cubical geometry, Noskov shows that one can take 2 as (synchronous) fftp constant.

\begin{lem}\label{lem:proj-median}
Let  $\ga$ be the $1$-skeleton of a locally finite CAT(0) cube complex.  Let $Z$ be a $\sigma$-quasi-convex subset of $\ga$. Then $Z$ has bounded projections. 
\end{lem}
\begin{proof}
For $x,y\in V\ga$, let $I[x,y]$ denote the set of vertices that appear in combinatorial geodesics from $x$ to $y$.
Note that $\ga$ is a median graph \cite[Theorem 6.1.]{Chepoi}, that is given any 3 vertices $x,y,z\in V\ga$,  the intersection $I(x,y)\cap I(x,z)\cap I(z,y)$ consist on a single vertex denoted $\mu(x,y,x)$ and called the {\it median} of $x,y,z$.

Let $x,y\in V\ga$, $d_\ga(x,y)=1$ and $z_x\in\pi_Z(x)$, $z_y\in \pi_Z(y)$. 
Let $m_x$ be the median of $x,z_x,z_y$ and $m_y$ the median of $y,z_y,z_x$. Since $m_x$ is in a geodesic from $z_x$ to $z_y$, $\d_\ga(m_x,Z)\leq \sigma$. 
Since $z_x\in Z$ and $m_x$ is in a geodesic path that realizes the minimum  distance from  $x$ to $Z$, we have that $d_\ga(m_x,z_x)\leq \sigma$. 
Similarly, $\d_\ga(m_y,z_y)\leq \sigma$.

Now, the median map $\mu$ is 1-Lipschitz \cite[Corollary 2.15]{ChaDruHag}, and thus since $d(x,y)=1$, we have that $\d(\mu(x,z_x,z_y),\mu(y,z_x,z_y))\leq 1$ and thus $\d(z_x,z_y)\leq 1+2 \sigma$.
\end{proof}

\subsection{Parabolic subgroups of right angled Artin groups}

Recall that a right angled Artin group $G$ has a standard presentation $\prs{X}{R}$ in which the only relations of $R$ consist on commuting relations among the elements of $X$. A parabolic subgroup $P\leqslant G$ is a subgroup that is $G$-conjugate to a subgroup $\gen{Y}$ where $Y$ is some subset of $X$. The Cayley graph of  $\ga(G,X)$ is the 1-skeleton of the Salvetti Complex, that is a locally finite CAT(0) cube complex. Therefore $\ga(G,X)$ has fftp. By \cite[Lemma 3.6]{HsuWise}, $P$ is quasi-convex in $\ga(G,X)$.

Thus we obtain the following corollary of Lemma \ref{lem:proj-median}, which gives another interesting family of examples for  Theorem \ref{thm:mainB}.
\begin{cor}
Let $G$ be a finitely generated  right angled Artin group, $X$ the standard generating set, and $P$ a parabolic subgroup of $G$. Then $\ga(G,X)$ has the falsification by fellow traveler property and $P$ has bounded projections in $\ga(G,X)$.
\end{cor}

\subsection{Standard parabolic subgroups of Coxeter groups}

Let $S$ be a finite set.
A \emph{Coxeter matrix} over $S$ is a square matrix $M=(m_{s,t})_{s,t \in S}$ with coefficients in $\N  \cup \{ \infty \}$ such that  $m_{s,s}=1$ for all $s \in S$ and $m_{s,t} = m_{t,s} \ge 2$ for all $s,t \in S$, $s \neq t$. The \emph{Coxeter group associated to $M$}  is the group given by the following presentation:
\[
W = \langle S \mid s^2=1 \text{ for all } s \in S\,, (st)^{m_{s,t}}=1 \text{ for all } s,t \in S \text{ such that } s \neq t \text{ and } m_{s,t} \neq \infty \rangle\,.
\]

Let $X$ be a subset $S$.
We denote by $W_X$ the subgroup of $W$ generated by $X$.
By Bourbaki \cite{Bourb1}, $W_X$ is a Coxeter group (associated to the submatrix of $M$ corresponding to $X$).
It is called a \emph{standard parabolic subgroup} of $W$. We recall now some standard facts about Coxeter groups.

We say that $w \in W$ is \emph{$X$-reduced} if it is of minimal length in its coset $W_Xw$.

\begin{lem}\label{lem:X-reduced}(Bourbaki \cite{Bourb1})
\begin{itemize}
\item[(1)]
There exists a unique $X$-reduced element in each coset $c \in W_X \backslash W$.
\item[(2)]
Let $u \in W$.
Then $u$ is $X$-reduced if and only if $|su|_S = |u|_S+1$ for all $s \in X$.
\item[(3)]
Let $w \in W$.
Let $u$ be the (unique) $X$-reduced element lying in $W_Xw$ and let $v \in W_X$ such that $w=vu$.
Then $|w|_S = |v|_S + |u|_S$.
\end{itemize}
\end{lem}

We will also make use of the ``Exchange condition'' that characterizes the Coxeter groups (see \cite[Chapter 4]{Davis1} for example).

\begin{lem}\label{lem:exchangecond}
Let $w \in W$ and let $s,t \in S$ such that $|sw|_S=|wt|_S=|w|_S+1$ and $|swt|_S<|w|_S+1$. 
Then $sw=wt$.
\end{lem}

The following proof was provided by Luis Paris.
\begin{prop}\label{prop:Luis} 
Let $W$ be a Coxeter group, $S$ a set of Coxeter generators and  $X\subseteq S$. Then $W_X$ has $1$-bounded projections in $\ga(W,S)$.
\end{prop}

\begin{proof}
Let $w \in W$.
Write $w= vu$ where $u$ is $X$-reduced and $v \in W_X$.
By Lemma \ref{lem:X-reduced}, for all $v' \in W_X$, we have 
\[
d_S (v',w) = |{v'}^{-1}vu|_S = |{v'}^{-1}v|_S + |u|_S \ge |u|_S = d_S(v,w)\,,
\]
and we have equality if and only if $|{v'}^{-1}v|_S=0$, that is, if and only if $v'=v$.
So, $\pi_{W_X} (w) = \{ v\}$.

Let $w' \in W$ such that $d_S(w',w)=1$.
Let $t \in S$ such that $w' =wt$.
Upon exchanging $w'$ and $w$ we may assume that $|w'|_S = |wt|_S = |w|_S+1$.
If $ut$ is $X$-reduced, then, by the above, $\pi_X(w') = \{v\}$ (since $wt = vut$) and the diameter of $\pi_{W_X}(w) \cup \pi_{W_X}  (w') = \{v\}$ is $0$.
So, we may assume that $ut$ is not $X$-reduced.
By Lemma \ref{lem:X-reduced} there exists $s \in X$ such that $|sut|_S < |ut|_S= |u|_S+1$.
We also have $|ut|_S = |u|_S+1$ by hypothesis and $|su|_S = |u|_S +1$ since $u$ is $X$-reduced.
By Lemma \ref{lem:exchangecond} this implies that $su = ut$, hence $w' = vsu$.
Since $u$ is $X$-reduced, it follows that $\pi_{W_X} (w') = \{ vs \}$, hence the diameter of $\pi_{W_X} (w) \cup \pi_{W_X} (w') = \{v, vs\}$ is $1$.
\end{proof}

\noindent{\textbf{{Acknowledgments}}}
This work was originated by a question of Ian Leary, who asked the author about growth Schreier graphs of relatively hyperbolic groups with respect to  parabolic subgroups, in the ferry returning from the conference in Ventotene 2015. The relative falsification by fellow traveler property appeared in the paper after a question of Ilya Gekhtman about the case of automatic groups. The author thanks Luis Paris for his interest and for providing the proof of Proposition \ref{prop:Luis}.
Finally, the authors thanks an anonymous referee for valuable comments that improved the exposition of the paper.


\end{document}